\documentclass[11pt,a4paper,reqno]{amsart}
\usepackage{amsmath,amssymb,amsfonts,epsfig,mathrsfs,cite}
\usepackage[T1]{fontenc}
\usepackage{color}
\usepackage{array}
\usepackage{amsthm}
\usepackage{amstext}
\usepackage{graphicx}
\usepackage{setspace}
\usepackage{floatrow} 
\usepackage[title]{appendix}

\usepackage{hyperref}

\usepackage{subfigure}
\makeatletter
\@namedef{subjclassname@2020}{%
	\textup{2020} Mathematics Subject Classification}
\makeatother

\usepackage[margin=2.5cm]{geometry}
\usepackage{color}
\usepackage{enumitem}
\usepackage{amscd,psfrag}
\usepackage{yhmath}
\usepackage[mathscr]{eucal}

\usepackage{comment}

\allowdisplaybreaks[4]

\usepackage{slashed}

\makeatletter
\pdfpageheight\paperheight
\pdfpagewidth\paperwidth

\usepackage{epstopdf}
\usepackage{chngcntr}
\counterwithin{figure}{section}
\usepackage{mathrsfs}

\usepackage{indentfirst}	

\usepackage[normalem]{ulem}
\theoremstyle{plain}

\newtheorem{definition}{Definition}
\newtheorem{theorem}[definition]{Theorem}
\newtheorem{notation}[definition]{Notation}
\newtheorem*{theorem*}{Theorem}

\newtheorem{remark}[definition]{Remark}

\newtheorem*{remark*}{Remark}
\newtheorem*{sideremark*}{Side Remark}

\newtheorem*{claim*}{Claim}
\newtheorem*{lemma*}{Lemma}
\newtheorem*{q*}{Question}
\newtheorem{lemma}[definition]{Lemma}

\newtheorem*{corollary*}{Corollary}

\newtheorem{proposition}[definition]{Proposition}

\newcommand{\R}{\mathbb{R}}

\newcommand{\na}{\nabla}
\newcommand{\emb}{\hookrightarrow}

\newcommand{\id}{\text{Id}}

\newcommand{\p}{\partial}

\newcommand{\weak}{\rightharpoonup}
\newcommand{\e}{\varepsilon}

\newcommand{\dd}{{\rm d}}

\newcommand{\linf}{{L^\infty}}

\newcommand{\bra}{\left\langle}
\newcommand{\ket}{\right\rangle}

\newcommand{\D}{{\mathbb{D}(u)}}
\newcommand{\dv}{{\rm div}}
\newcommand{\tor}{{\mathbb{T}^2}}
\newcommand{\trho}{{\tilde{\rho}}}
\newcommand{\divu}{{{\rm div}u}}
\newcommand{\rotu}{{{\rm rot}u}}
\newcommand{\hatrho}{{\widehat{\rho}}}

\newcommand{\rhod}{\left(\rho^\delta\right)}

\newcommand{\rb}{\mathfrak{b}}

\def\XXint#1#2#3{{\setbox0=\hbox{$#1{#2#3}{\int}$ }
		\vcenter{\hbox{$#2#3$ }}\kern-.6\wd0}}

\newcommand{\limphi}{{\bra{\Phi}\ket}}

\newcommand{\rhos}{{\bra \rho^2\ket}}

\title{On global existence and large-time behaviour of weak solutions to the compressible barotropic Navier--Stokes Equations  on $\mathbb{T}^2$ with density-dependent bulk viscosity: beyond the Va\u{\i}gant--Kazhikhov regime}

\author{Siran Li}

\address{Siran Li: School of Mathematical Sciences $\&$ CMA-Shanghai, Shanghai Jiao Tong University, No.~6 Science Buildings,
	800 Dongchuan Road, Minhang District, Shanghai, China (200240)}

\email{\texttt{siran.li@sjtu.edu.cn}}

\author{Jianing Yang}

\address{Jianing Yang: School of Mathematical Sciences, Shanghai Jiao Tong University, No.~6 Science Buildings,
	800 Dongchuan Road, Minhang District, Shanghai, China (200240)}

\email{\texttt{jnyang22@sjtu.edu.cn}}


\subjclass[2020]{}
\date{\today}

\begin{document}

\begin{abstract}
We are concerned with the compressible barotropic Navier--Stokes equations for a $\gamma$-law gas with density-dependent bulk viscosity coefficient $\lambda=\lambda(\rho)=\rho^\beta$ on the two-dimensional periodic domain $\mathbb{T}^2$. The global existence of weak solutions with initial density bounded away from zero and infinity for $\beta>3$, $\gamma>1$ has been established by Va\u{\i}gant--Kazhikhov [\textit{Sib. Math. J.} 36 (1995), 1283--1316]. When $\gamma=\beta>3$, the large-time behaviour of the weak solutions and, in particular, the absence of formation of vacuum and concentration of density as $t \to \infty$, has been proved by Perepelitsa [\textit{SIAM J. Math. Anal.} 39 (2007/08), 1344--1365]. Huang--Li [\textit{J. Math. Pures Appl.} 106 (2016), 123--154] extended these results by establishing the global existence of weak solutions and large-time behaviour under the assumptions  $\beta >3/2$, $1< \gamma<4\beta-3$, and that the initial density stays away from infinity (but may contain vacuum).

Improving upon the works listed above, we prove that in the regime of parameters as in Huang--Li, namely that $\beta >3/2$ and $1< \gamma<4\beta-3$, if the density has no vacuum or concentration at $t=0$, then it stays away from zero and infinity at all later time $t \in ]0,\infty[$. Moreover, assuming $\beta>1$, $\gamma>1$ and a technical condition, we establish the global existence of weak solutions  on $\mathbb{T}^2$. One of the key ingredients of our proof is a novel application --- motivated by the recent work due to Danchin--Mucha [\textit{Comm. Pure Appl. Math.} 76 (2023), 3437--3492]  --- of Desjardins' logarithmic interpolation inequality.

	\end{abstract}
	\maketitle

	\section{Introduction}\label{sec: intro}

We are concerned with the global-in-time existence and large-time behaviour of solutions for viscous compressible fluids in the barotropic regime. The partial differential equations (PDE) describing the motion for  barotropic viscous compressible fluids are the Navier--Stokes system:
\begin{equation}\label{PDE, 1}
    \begin{cases}
        \rho_t + {\rm div} (\rho u)=0,\\
        	(\rho u)_t+\mathrm{div}(\rho u\otimes u)={\rm div}\sigma.
    \end{cases}
\end{equation}
In the above, $\rho=\rho(t,x)$ and $u=u(t,x)$ are the density and fluid velocity, respectively; $\sigma$ is the stress tensor of the fluid obeying the Stokes' law:
\begin{equation*}
    \sigma = \mathcal{S} - P\id.
\end{equation*}
The pressure $P=P(\rho)$ is a scalar function depending only on the density in the barotropic regime, and the viscous stress tensor $\mathcal{S}$ measures the resistance of the fluid to flow. Changes in temperature and external forces are not taken into account.

The mathematical theory of multidimensional barotropic viscous compressible flows has been an active field of research in the past three decades. Foundational contributions have been made, among many others, by P.-L. Lions \cite{lions} (1993--1998) and E. Feireisl \emph{et al} \cite{fnp} (2001). In the Lions--Feireisl theory, the stress tensor takes the specific form:
\begin{align*}
    \sigma = 2\mu \,\D + \lambda\, \dv u \,\id - P(\rho)\id.
\end{align*}
Here $\D := \frac{1}{2}\left(\na u + \na^\top u\right)$ is the rate-of-strain tensor; $\lambda$ and $\mu$ are the bulk and shear viscosity coefficients, respectively, which satisfy
\begin{align*}
    \mu > 0 \qquad
     \text{and}\qquad \lambda + \frac{2}{d}\mu>0,
\end{align*}
where $d$ is the dimension of the spatial domain.  Throughout this work, we consider Equation~\eqref{PDE, 1} on $\mathbb{T}^2=\mathbb{R}^2/\mathbb{Z}^2$ (namely, $d=2$) and, in line with the seminal paper \cite{vaigant} by Va\u{\i}gant--Kazhikhov (1995), we work under the additional assumptions:
\begin{equation}\label{assumption}
0<\mu = \mu(\rho)={\rm constant}, \quad \lambda(\rho) = \rho^\beta,\quad P(\rho) = \rho^\gamma.
\end{equation}
The fluid in consideration is a $\gamma$-law barotropic gas with bulk viscosity depending on the density. The constants in front of $\rho^\beta$ in $\lambda$ and $\rho^\gamma$ in $P$ are normalised to $1$. In this case, the initial value problem for Equation~\eqref{PDE, 1} becomes 	\begin{equation}\label{equ}
		\begin{cases}
			\rho_t+\mathrm{div}(\rho u) = 0,\\
			(\rho u)_t+\mathrm{div}(\rho u\otimes u)= \mathcal{D}-\nabla P\qquad \text{in } [0,T] \times \tor,\\
   u(0,x)=u_0(x),\quad \rho(0,x)=\rho_0(x)\qquad\text{for } x\in \mathbb{T}^2,
		\end{cases}
	\end{equation} 
with the isotropic stress tensor $\mathcal{D}$ given by
\begin{align*}
\mathcal{D}:=\nabla\big((\lambda+\mu)\mathrm{div} \,u\big)+\mu\Delta u.
\end{align*}

It should be highlighted that in \cite{vaigant} by Va\u{\i}gant--Kazhikhov (1995), the following range of  coefficients is assumed: 
\begin{align*}
    \beta > 3 \qquad \text{ and }\qquad \gamma>1.
\end{align*}

The study of Equation~\eqref{equ} abounds in the literature. A brief survey is given below, though the list of references is by no means exhaustive. Note first that the barotropic Navier--Stokes in one spatial dimension (1D) is a classical topic whose developments predate the Lions--Feireisl theory. See Kazhikhov--Shelukhin \cite {kazhi}, D. Serre \cite{serre1, serre2}, Hoff \cite{hoff-1d, hoff1}, and Beir\~{a}o da Veiga \cite{bdv}, etc. The case of density-dependent viscosity coefficient in 1D has been treated by Mellet--Vasseur \cite{mv}.

In the multidimensional case, the well-posedness theory of Equation~\eqref{equ} has been studied extensively in the case that both the shear and bulk viscosities are positive constants. The classical papers by Serrin \cite{serrin} and Nash \cite{nash} established the local existence and uniqueness of classical solutions when the initial density $\rho_0$ is away from the vacuum. This has been extended to strong solutions and to the case that the initial density is allowed to vanish on nontrivial open sets. See, \emph{e.g.}, Solonnikov \cite{sol}, Salvi--Stra\v{s}kraba \cite{ss}, Cho--Choe--Kim \cite{cck}, and Choe--Kim \cite{ck}. The global existence of  solutions was first obtained by Matsumura--Nishida \cite{mn} for classical solutions with initial data close to a non-vacuum equilibrium, and then by Hoff \cite{hoff2} for weak solutions (or strong solutions away from the ``initial layer'') with discontinuous initial data. The global existence of weak solutions with large data, in contrast, is the theme of the celebrated Lions--Feireisl theory of renormalised solutions as aforementioned \cite{lions, fnp, feireisl}. For the global existence of classical and strong solutions with smooth data of small energy and possibly vanishing initial density, see Huang--Li \cite{H2}, Huang--Li--Xin \cite{xin1}, Jiu--Wang--Xin \cite{xin2}, Li--Liang \cite{ll} and the references cited therein. All the results mentioned in this paragraph also hold on the whole space $\R^d$ ($d\ \in \{2,3\}$) under suitable far-field decay conditions.

On the other hand, the case of density-dependent coefficients of the isotropic stress tensor $\mathcal{D}$  poses enormous challenges to the analysis of Equation~\eqref{equ}. The only successful insights obtained in this case before 2018 or so were reported in Bresch--Desjardins \cite{bd, bd'}, Bresch--Desjardins--G\'{e}rard-Valet \cite{bdg}, Bresch--Desjardins--Zatorska \cite{bdz}, Li--Xin \cite{lx}, and Vasseur--Yu \cite{vy}. These works typically assume the following condition on  the viscosity coefficients:
\begin{align}\label{BD relation}
    \lambda(\rho) = 2\rho \mu'(\rho) -2\mu(\rho),
\end{align}
under which the BD-estimate (\emph{\`{a} la} Bresch--Desjardins \cite{bd'}) is valid. 
See also the review paper \cite{bd-handbook} for a detailed account. In the important work \cite{bj}, Bresch--Jabin (2018) established, by way of exploring novel compactness estimates for the continuity equation, the global existence of weak solutions for the barotropic Navier--Stokes system in multidimensions with general, thermodynamically unstable pressure laws and the anisotropic viscous stress tensor $\mathcal{D}$.

Our current paper is closely related to the now-classical work \cite{vaigant} by Va\u{\i}gant--Kazhikhov (1995), which establishes the global existence of classical, strong, and weak solutions for Equation~\eqref{equ} on $\tor$ under the assumption~\eqref{assumption} --- summarised as Theorem~\ref{thm: VK} below. Note that the viscosity coefficients as in assumption~\eqref{assumption} do \emph{not} satisfy the condition~\eqref{BD relation} for  BD-estimates.

\begin{theorem}[Va\u{\i}gant--Kazhikhov \cite{vaigant}]
\label{thm: VK}
Assume the condition~\eqref{assumption} and $\beta>3$, $\gamma>1$.\footnote{In fact, here it suffices to assume $\gamma \geq 0$ in place of $\gamma>1$.}  Then
\begin{enumerate}
    \item 
    Let $\left(\rho_0, u_0\right)$ be such that
		$$0<m<\rho_0(x)<M<+\infty,\quad x\in\mathbb{T}^2,$$
		$$(\rho_0,u_0)\in C^{1+\omega}\left(\mathbb{T}^2\right)\times C^{2+\omega}\left(\mathbb{T}^2;\R^2\right),\quad 0<\omega<1$$ for some $m,M$. Then there exists a unique global classical solution for Equation~\eqref{equ} such that
		$$\rho \in C^{1+\omega/2,2+\omega}\left([0,\infty[\times \mathbb{T}^2\right),$$
		$$u\in C^{1+\omega/2,2+\omega}\left([0,\infty[\times \mathbb{T}^2; \R^2\right).$$
The density is bounded away from vacuum for all finite time. 
\item 
Let  $\left(\rho_0, u_0\right)$ be such that
		$$0<m<\rho_0(x)<M<+\infty,\quad x\in\mathbb{T}^2,$$
		$$(\rho_0,u_0)\in W^{1,q}\left(\mathbb{T}^2\right) \times H^2\left(\mathbb{T}^2;\R^2\right),\quad q>2$$
		for some $m, M$. Then there exists a global strong solution for Equation~\eqref{equ}. The density $\rho$ is an $L^{\infty}$-function essentially bounded away from vacuum for all finite time.
  \item 
Let  $\left(\rho_0, u_0\right)$ be such that 
\begin{align*}
    \left(\rho_0, u_0\right) \in L^\infty(\tor) \times H^{1}\left(\tor;\R^2\right).
\end{align*}
Then there exists a global weak solution for Equation~\eqref{equ}.
  
\end{enumerate}
	\end{theorem}

Throughout we adhere to the following convention in \cite[p.1109]{vaigant}:
\begin{itemize}
    \item 
    $(\rho, u)$ is a weak solution to Equation~\eqref{equ} if it satisfies the PDE in the distributional sense;
    \item 
     $(\rho, u)$ is a strong solution if all of its derivatives are regular distributions and Equation~\eqref{equ} also holds in the \emph{a.e.} sense; 
     \item 
$(\rho, u)$ is a classical solution if  all the terms in Equation~\eqref{equ} are H\"{o}lder continuous.
\end{itemize}

With Theorem~\ref{thm: VK} established, one naturally asks about the large-time behaviour of solutions for Equation~\eqref{equ} in the Va\u{\i}gant--Kazhikhov regime. The following result was obtained by Perepelitsa \cite{perep} (2006). Here and hereafter, all the norms are understood as taken over $\tor$ unless otherwise specified; \emph{e.g.}, $\|f\|_{L^q}=\|f\|_{L^q(\mathbb{T}^2)}$, $\|f\|_{W^{k,q}}=\|f\|_{W^{k,q}(\mathbb{T}^2)}$, and $\|f\|_{H^k}=\|f\|_{W^{k,2}(\mathbb{T}^2)}$. The time variable $t$ is usually suppressed.

\begin{theorem}[Perepelitsa \cite{perep}]\label{thm: perepelitsa}
Under the assumption~\eqref{assumption} and that $\beta > 3, \gamma>1$, suppose that
\begin{eqnarray*}
    && \left(\rho_0, u_0\right) \in L^\infty(\tor) \times H^{1}\left(\tor;\R^2\right),\\
    && 0<m < \mathop{\rm ess\,inf}_{\tor} \rho_0 \leq \mathop{\rm ess\,sup}_{\tor} \rho_0 < M < +\infty
\end{eqnarray*}
for some $m,M$. Then there exists a global weak solution for Equation~\eqref{equ} such that
\begin{align*}
    (\rho, u) \in C\left([0,\infty[; L^2(\tor)\right) \times C\left([0,\infty[; L^2_w\left(\tor;\R^2\right)\right),
\end{align*}
and that for any $T>0$ there are $0<\bar{m}(T)<\bar{M}(T)<\infty$ such that 
\begin{align*}
    \bar{m}(T) \leq \rho(t,x) \leq \bar{M}(T) \qquad \text{for each } t \in ]0,T[ \text{ and a.e. } x \in \tor.
\end{align*}

Moreover, if $\beta=\gamma>3$, then there exists a weak solution such that 
\begin{align*}
    \lim_{t \to \infty} \left\{ \left\| \rho(t,\bullet) - \int_\tor \rho_0(y)\,\dd y\right\|_{\linf} + \|u(t,\bullet)\|_{L^p}\right\} = 0\qquad \text{for any } p>1.
\end{align*}
\end{theorem}

Theorem~\ref{thm: perepelitsa} ascertains that, for initial density bounded away from $0$ and $+\infty$, the global weak solution (whose existence is guaranteed by Theorem~\ref{thm: VK}, (3)) develops neither vacuum nor concentration up to any finite time. In general, however, it does not rule out the formation of vacuum or concentration at $T=+\infty$. The large-time behaviour is only established when $\beta=\gamma>3$. The proof in \cite{perep} utilises, among other more ``standard'' tools, commutator estimates and Poincar\'{e}--Sobolev inequalities in Orlicz spaces.

Based on developments in Jiu--Wang--Xin \cite{xin2} on removing the condition ${\rm ess\,inf}_{\tor} \rho_0 >m >0$ in Theorem~\ref{thm: perepelitsa}, Huang--Li \cite{H1}  (2016) established Theorem~\ref{thm: Huang-Li} below (among other results), which relaxes the condition for the large-time behaviour from $\gamma=\beta>3$ in 
Theorem~\ref{thm: perepelitsa} to $ \beta>3/2$ and $1<\gamma < 4\beta-3$. In comparison to Theorem~\ref{thm: main}, the result in \cite{H1} allows for the existence of vacuum at any time. More precisely:
\begin{theorem}[Theorem~1.2 in \cite{H1}]\label{thm: Huang-Li}
Suppose that $\beta>4/3$, $\gamma>1$ and that
$$\left(\rho_0, u_0\right) \in L^\infty\left(\tor\right) \times H^{1}\left(\tor;\R^2\right) \qquad\text{and}\qquad \rho_0 \geq 0.$$
Then Equation~\eqref{equ} has a global weak solution $(\rho, u)$ in $]0,\infty[ \times \tor$ such that for any $0<\tau < T< \infty$ and $p \geq 1$, one has
\begin{eqnarray*}
    && \rho \in L^\infty\left(0,T; \linf\left(\tor\right)\right) \cap C\left(0,T; L^p\left(\tor\right)\right),\\
    && u \in L^\infty\left(0,T; H^1\left(\tor;\R^2\right)\right), \quad u_t \in L^2\left(\tau, T; L^2\left(\tor;\R^2\right)\right), \quad \na u \in \linf\left(\tau,T; L^p\left(\R^{2\times 2}\right)\right). 
\end{eqnarray*}
If, in addition, $\beta>3/2$ and $1<\gamma < 4\beta-3$,  then  it holds that
\begin{align*}
    \sup_{t \in [0,\infty[} \|\rho(t,\bullet)\|_\linf \leq C\left(\mu,\beta,\gamma,\|\rho_0\|_\linf,\|u_0\|_{H^1}\right)
\end{align*}
and that
\begin{align*}
    \lim_{t \to \infty} \left\{ \left\| \rho(t,\bullet) - \int_\tor \rho_0(y)\,\dd y\right\|_{L^p} + \|\na u(t,\bullet)\|_{L^p}\right\} = 0\qquad \text{for any } p \geq 1.
\end{align*}
\end{theorem}

Based on Theorems~\ref{thm: VK}, \ref{thm: perepelitsa}, and \ref{thm: Huang-Li}, we obtain the global existence of weak solutions for Equation~\eqref{equ}, provided that
\begin{align*}
    \beta>1 \qquad \text{ and }\qquad \gamma>1,
\end{align*}
under an additional technical assumption (see~$(\clubsuit)$ in Theorem~\ref{thm: main} below). We rule out the possibility of vacuum formation at any finite time for weak solutions, provided that the initial density stays essentially away from vacuum or concentration.

Our main theorem is as follows:
\begin{theorem}\label{thm: main}
Suppose that $\beta>1$ and $\gamma>1$ and the initial data satisfy
\begin{eqnarray*}
    && \left(\rho_0, u_0\right) \in L^\infty(\tor) \times H^{1}\left(\tor;\R^2\right),\\
    && 0<m < \mathop{\rm ess\,inf}_{\tor} \rho_0 \leq \mathop{\rm ess\,sup}_{\tor} \rho_0 < M < +\infty
\end{eqnarray*}
for some $m,M$. Suppose, in addition, that
\begin{align*}
&\text{There exist an index $q_1>2$, smooth approximations $\{(\rho_n, u_n)\}$ of the solutions, and}\\
&\text{a constant $C_\star$ uniform in $n$ such that } \sup_{n \in \mathbb{N}}\,\sup_{t \geq 0}\, \left\|\rho_n(t,\bullet) u_n(t,\bullet)\right\|_{L^{q_1}(\tor)} \leq C_\star.  \tag{$\clubsuit$}
\end{align*}
Then there exists at least one weak solution $(\rho, u)$ of Equation~\eqref{equ} up to any finite time $T$, satisfying that
\begin{align*}
0< \bar{m}(T) \leq \rho(t,x) \leq \bar{M}(T) < \infty \qquad \text{for each } T\in ]0,+\infty[ \text{ and a.e. } (t,x) \in [0,T]\times \tor,
\end{align*}
and that the global weak solutions satisfy the large-time behaviour:
\begin{align*}
    \lim_{t \to \infty} \left\{ \left\| \rho(t,\bullet) - \int_\tor \rho_0(y)\,\dd y\right\|_{L^{p}} + \|\na u(t,\bullet)\|_{L^p}\right\} = 0 \qquad \text{for any } p \geq 1.
\end{align*}

In the case $\beta>{3}/{2}$ and $1<\gamma<4\beta-3$, \emph{without assuming the condition~($\clubsuit$)}, there exist a uniform constant $\bar{M}$ independent of time and a constant $\bar{\bar{m}}(T)$ depending on $T$, such that
\begin{align*}
0 < \bar{\bar{m}}(T) \leq \rho(t,x) \leq \bar{M} < \infty \qquad \text{for each } t \in [0,T]  \text{ and a.e. } x \in \tor, 
\end{align*}
and that the global weak solutions satisfy the large-time behaviour:
\begin{align*}
    \lim_{t \to \infty} \left\{ \left\| \rho(t,\bullet) - \int_\tor \rho_0(y)\,\dd y\right\|_{L^{p}} + \|\na u(t,\bullet)\|_{L^2}\right\} = 0 \qquad \text{for any } p \geq 1.
\end{align*}
\end{theorem}

\begin{remark}
We may take $\bar{M} = \bar{M}\big(\mu, \gamma, \beta, M, E_0\big)$, $\bar{\bar{m}}(T) = \bar{\bar{m}}\big(T, \mu, \gamma, \beta, M, m, E_0\big)$, $\bar{m}(T) = \bar{m}\big(T,\mu, \gamma, \beta, M, m, E_0,q_1, C_\star\big)$, and $\bar{M}(T) = \bar{M}\big(T, \mu, \gamma, \beta, M, E_0,q_1, C_\star\big)$. Here, $E_0$ is the total initial energy (see Lemma~\ref{lemma: energy ineq} below):
\begin{align*}
    E_0:=\int_{\mathbb{T}^2} \left\{\frac{\rho_0(x)|u_0(x)|^2}{2} +\frac{\big(\rho_0(x)\big)^\gamma}{\gamma-1}\right\}\,\mathrm{d}x,
\end{align*}
and $q_1, C_\star$ are as in condition~$(\clubsuit)$.
\end{remark}

\begin{remark}
The key novelty of Theorem~\ref{thm: main} --- in particular, in comparison with Theorem~\ref{thm: Huang-Li} by Huang--Li~\cite{H1} --- is the non-formation of vacuum in any finite time for the weak solutions, provided that the initial datum does not contain vacuum. 
\end{remark}


Our strategy for proving Theorem~\ref{thm: main} is largely motivated by the recent important work by Danchin--Mucha \cite{Danchin} (2023), which established the global existence of strong solutions for two-dimensional compressible Navier--Stokes equations with arbitrarily large initial velocity and almost constant density under the assumption of large bulk viscosity. In particular, we obtain from \cite{Danchin} the insight that the logarithmic interpolation inequality \emph{\`{a} la}
Desjardins \cite{Desjardins} may serve as the key ingredient for the proof.

Before further development, we first report some notations  used throughout this paper.
	$$\left\{
	\begin{aligned}
		& \partial_i=\frac{\partial}{\partial x_i},\quad \bar{f}=\int_{\mathbb{T}^2} f(t,x)\,\mathrm{d}x,\quad f_{U}=\frac{1}{|U|}\int_{U} f(t,x)\,\mathrm{d}x,\\
		&\nabla=(\partial_1,\partial_2),\quad \nabla^{\bot}=(\partial_2,-\partial_1),\\
		&\mathrm{div}u=\nabla\cdot u,\quad \mathrm{rot} u=\nabla^{\bot}\cdot u,\\
		&\frac{\mathrm{D}}{\mathrm{D}t} f \equiv \dot{f} := \frac{\partial}{\partial t}f +u\cdot \nabla f,\quad \tor=\left[-\frac{1}{2},\frac{1}{2}\right]^2\subset \mathbb{R}^2.
	\end{aligned}
	\right.$$
Also, unless otherwise specified, the Sobolev norms are taken only with respect to the spatial variable. That is, we understand $\|\na u\|_{L^2}$ as $\|\na u(t,\cdot)\|_{L^2}$, and similarly for the other norms.

The major steps for the proof of Theorem~\ref{thm: main} are outlined as follows. 
\begin{itemize}
    \item 
First, as per \cite{perep, lions, fnp, vaigant}, we rewrite Equation~\eqref{equ} in the form of an evolution equation, with the source term equal to the commutator
\begin{align}\label{G, def}
G := \sum_{i,j \in \{1,2\}}    \big[u_i,(-\Delta)^{-1}\partial_i\partial_j\big](\rho u_j) \equiv \sum_{i,j \in \{1,2\}} \big[u_i, R_iR_j\big](\rho u_j),
\end{align}
where $R_i$ denotes the $i^{\text{th}}$ Riesz transform. See Equations~\eqref{eq1} and \eqref{ui} below.

\item 
Then, by the theory of compensated compactness (see, \emph{e.g.}, Coifman--Lions--Meyer--Semmes \cite{Coi}, Coifman--Rochberg--Weiss\cite{Coifman}, and Coifman--Meyer \cite{Coifman1}), we may bound $\|\na u\|_{L^2}$ by the exponential function of powers of $\|\rho\|_\linf$. More precisely, it is shown that  $\log\left(e^2+\|\nabla u\|_{L^2}^2\right)$ does not exceed a polynomial function of $\|\rho\|_{L^{\infty}}$. See Proposition~\ref{propn: logY} below. Similar arguments can be found in Huang--Li \cite{H1} and Perepelitsa \cite{perep}.

\item 
The essentially novel ingredient of this paper lies in deriving the uniform-in-time upper and lower bounds for $\|\rho\|_\linf$. The key here is to estimate $\|G\|_\linf$, which, thanks to the usual Gagliardo--Nirenberg--Sobolev inequality, amounts to controlling $\left\|\rho^{\frac{1}{q}}u \right\|_{L^{q}}$.

To this end, we employ Lemma~\ref{rhou}, which is a weighted endpoint Sobolev inequality of the Brezis--Wainger type \emph{\`{a} la} B. Desjardins \cite{Desjardins}, to obtain that $$\left\|\rho^{\frac{1}{q}}u \right\|_{L^{q}} \lesssim \|\na u\|_{L^2}^{1-\frac{2}{q}} \times \log\left(2+\|\na u\|^2_{L^2}\right)^{\frac{1}{2}-\frac{1}{q}}\qquad\text{for any } q>2.$$
See Proposition~\ref{propn: boundG} below for details. This, together with the estimates established in the previous steps, yields the conclusion via Gronwall's inequality. As aforementioned, these arguments are motivated by the recent important work by Danchin--Mucha \cite{Danchin} (2023). 
\end{itemize}



The remaining parts of the paper are organised as follows. In Section~\ref{sec: prelim}, we collect some elementary facts and inequalities which will be needed in later analysis. Section~\ref{sec: log Y} is devoted to deriving several important \emph{a priori} estimates for strong solutions, culminating in the bound for $\log Y$ in Proposition~\ref{propn: logY}. Here $$ Y^2(t) :=\int_{\mathbb{T}^2} \left\{\mu|\mathrm{rot} u(t,x)|^2+\frac{B^2(t,x)}{\lambda(t,x)+2\mu}\right\}\,\mathrm{d}x,$$ in which $B$ is the effective viscous flux: $$B(t,x):=(\lambda(t,x)+2\mu )\mathrm{div} u(t,x) -\big(P(t,x)-\bar{P}(t)\big).$$ Then, in Section~\ref{sec: G}, uniform bounds for the commutator term $G$ will be derived:
\begin{align*}
G := \sum_{i,j \in \{1,2\}}    \big[u_i,(-\Delta)^{-1}\partial_i\partial_j\big](\rho u_j) \equiv \sum_{i,j \in \{1,2\}} \big[u_i, R_iR_j\big](\rho u_j),
\end{align*}
where $R_i$ is the $i^{\text{th}}$ Riesz transform. The main result of our paper, Theorem~\ref{thm: main}, shall be proved in Sections~\ref{sec: rho bound} and \ref{sec: global existence} --- In Section~\ref{sec: rho bound} we work under the assumption $\beta>4/3;  
\gamma>1$ as in Huang--Li \cite{H1} (but with initial density bounded below away from zero) and deduce the large-time behaviour of $\rho$ and $\na u$.   Finally, in Section~\ref{sec: global existence}, we establish the global existence of weak solutions up to any finite time $T \in ]0,\infty[$, under the more general assumption $\beta>1; \gamma>1$.

Before concluding the introduction, we reiterate that  our work builds essentially upon  Va\u{\i}gant--Kazhikhov \cite{vaigant}, Perepelitsa \cite{perep}, Huang--Li \cite{H1}, and Danchin--Mucha  \cite{Danchin}. Nevertheless, various novel ideas and estimates are developed to achieve the proof of our Main Theorem~\ref{thm: main}.

The following questions would be interesting for future investigation: 
\begin{itemize}
\item
Can we get rid of the additional assumption~$(\clubsuit)$ in Theorem~\ref{thm: main}?
\item
Can we prove the global existence of weak solutions for the barotropic compressible Navier--Stokes equations, also in the regime $\beta>1$ and $\gamma>1$, for the initial density $\rho_0$ admitting vacuum (and with suitable compatibility conditions)?
\item
Can we rule out the possibility of the formation of vacuum at infinite time? That is, do we have ${\rm ess\,inf}_{]0,\infty[\times\tor}\rho(t,x)>0$?
\end{itemize}

\section{Preliminaries}\label{sec: prelim}

This section collects several analytic tools that shall be used in the later developments. Throughout, we write $C=C(a_1, a_2, \ldots, a_n)$ to denote that the constant $C$ depends only on parameters $a_1, a_2, \ldots, a_n$. It may change from line to line.
 
We start with the following interpolation inequality, which is an instance of the Gagliardo--Nirenberg--Sobolev inequalities (\emph{cf}. Ladyzhenskaya--Solonnikov--Ural'ceva \cite{lsu}; see also Huang--Li \cite[Lemma~2.2]{H1}):
	\begin{lemma}
		\label{1}
For each $2<q<\infty$, there exists a universal constant $C$ such that for each $f\in H^1(\mathbb{T}^2)$, the following estimate holds:
		\begin{equation*}
			\|f\|_{L^q} \leq C \sqrt{q}\|f\|_{L^2}^{\frac{2}{q}} \|f\|_{H^1}^{1-\frac{2}{q}}.
		\end{equation*}
	\end{lemma}

\begin{remark}\label{remark: const}
The concrete form of the constant --- $C\sqrt{q}$ for $C$ independent of $q$ --- shall play a crucial role in the proof of our Main Theorem~\ref{thm: main}; see, in particular, the estimate~\eqref{final, Nov24} in \S\ref{sec: final}. 
\end{remark}

The div-curl estimate below follows from standard elliptic estimates.
 \begin{lemma}
		\label{lemma1}
  Let $1<q<\infty$. For any vector-valued function $f=(f_1,f_2) \in L^q(\tor, \R^2)$ such that $F:=\left({\rm div}f, {\rm rot} f \right)$ belongs to $L^q(\tor, \R^2)$, it holds that $f \in W^{1,q}(\tor,\R^2)$. In addition, there exists a constant $C=C(q)$ such that 
		$$\Vert \nabla f\Vert_{L^q}\leq C\Vert F\Vert_{L^q}.$$
	\end{lemma}

Next, we recall the endpoint Sobolev inequality for $\linf$ \emph{\`{a} la} Brezis--Wainger \cite{BW}.
 
	\begin{lemma}\label{lemma: brezis-wainger}
		Let $q>2$. There exists a positive constant $C=C(q)$  such that for every function $f\in W^{1,q}(\mathbb{T}^2)$, we have that
		\begin{equation*}\|f\|_{L^{\infty}}\leq C\Big(\|\nabla f\|_{L^2} \sqrt{\log (e+\|\nabla f\|_{L^q})} +\|f\|_{L^2} +1\Big).
		\end{equation*}
	\end{lemma}

To control the deviation of a Sobolev function from its mean value, we shall make use of the Trudinger inequality. See \emph{e.g.}, \cite{ps}.
	\begin{lemma}
		\label{Tru}
		There are universal constants $c_1, c_2>0$ such that for all $f\in H^1(\mathbb{T}^2)$, it holds that
		$$\int_{\mathbb{T}^2} \exp\left(\frac{\left|f(x)-\bar{f}\right|^2}{c_1\|\nabla f\|_{L^2}^2}\right)\,\mathrm{d}x\leq c_2.$$
	\end{lemma}
	 Let $\mathcal{H}^1$ and $\mathcal{BMO}$ stand for the usual Hardy and BMO spaces defined over $\tor$:
	\begin{equation}
		\label{hardy}
		\mathcal{H}^1 =\left\{f\in L^{1}(\mathbb{T}^2): \|f\|_{\mathcal{H}^1}= \|f\|_{L^1}+\|R_1 f\|_{L^1} +\|R_2f\|_{L^1} <+\infty\right\},
	\end{equation}
	\begin{equation}
		\label{bmo}
		\mathcal{BMO} =\left\{ f\in L_{loc}^1(\mathbb{T}^2): \|f\|_{\mathcal{BMO}} <+\infty\right\}
	\end{equation}
	with $$\|f\|_{\mathcal{BMO}}=\sup\limits_{x\in\mathbb{T}^2, \,0<r<1}\frac{1}{|\Omega_r(x)|}\int_{\Omega_r(x)} \left|f(y)-f_{\Omega_r(x)}\right| \,\mathrm{d}y,$$
	where $\Omega_r(x) =\mathbb{T}^2\cap B_r(x)$. 
Fefferman's seminal result \cite{Feff} shows that $\mathcal{BMO}$ is the dual space of $\mathcal{H}^1$. Also, it follows from Lemma~\ref{Tru} and the definition of $\mathcal{BMO}$ that for any $f\in \mathcal{BMO}$, 
	\begin{equation}
		\label{bmo2}
		\|f\|_{\mathcal{BMO}}\leq C\|\nabla f\|_{L^2}.
	\end{equation}
 Here $C>0$ is a universal constant.

 In the subsequent developments, we also need the commutator estimates due to Coifman--Rochberg--Weiss\cite{Coifman} and Coifman--Meyer \cite{Coifman1}. Denote $(-\Delta)^{-1}$ the inverse of  Laplacian with zero mean on $\mathbb{T}^2$, and $$R_i=(-\Delta)^{-\frac{1}{2}}\partial_i,\qquad i \in \{1,2\}$$ the usual Riesz transforms on $\mathbb{T}^2$. The composition of Riesz transforms $R_i\circ R_j$ may be represented as a singular integral operator:
	\begin{equation*}
		R_i\circ R_j(f) (x)= {\rm p.v.}\int_{\mathbb{T}^2} K_{ij}(x-y) f(y)\,\mathrm{d}y.
	\end{equation*}
Here \emph{p.v.} denotes the principal value integral, and the kernel $K_{ij}(x)$ ($i,j \in \{1,2\}$ has a singularity of the second order at 0: $$|K_{ij}(x)|\lesssim |x|^{-2}$$ modulo a uniform constant. Then, given a function $g$, consider the linear operator:
	$$[g,R_iR_j](f) :=g R_i\circ R_j(f)-R_i\circ R_j(gf), \qquad i,j\in \{1,2\}.$$ It can be written as a convolution operator with the singular kernel $K_{ij}$:
	$$[g,R_iR_j](f)(x)= {\rm p.v.}\int_{\mathbb{T}^2} K_{ij}(x-y)\big(g(x)-g(y)\big) f(y)\,\mathrm{d}y.$$
	In view of the results of Coifman--Rochberg--Weiss\cite{Coifman} and Coifman--Meyer \cite{Coifman1}, this commutator operator can be estimated as follows:
	\begin{lemma}
		\label{commu}The following map
		$$ \begin{aligned}
			W^{1,r_1}(\mathbb{T}^2,\R^2)\times L^{r_2}(\mathbb{T}^2,\R^2) &\longrightarrow W^{1,r_3}(\mathbb{T}^2)\\
			(g,f)&\longrightarrow [g_j, R_iR_j]f_i,
		\end{aligned}$$
		is continuous when $\frac{1}{r_3}=\frac{1}{r_1}+\frac{1}{r_2}$. Indeed, there is a  positive constant $C=C(r_1,r_2,r_3)$ such that 
  \begin{align*}
      \big\|\nabla\big([g_j, R_iR_j]f_i\big)\big\|_{L^{r_3}}\leq C\| \nabla g\|_{L^{r_1}} \| f\|_{L^{r_2}}.
  \end{align*}
In addition, for each $1<q<\infty$,  there is a positive constant $C'=C'(q)$ such that	$$\big\|[g_j, R_iR_j]f_i\big\|_{L^q}\leq C'\| g\|_{\mathcal{BMO}} \Vert f\Vert_{L^q}.$$
	\end{lemma}

The following estimate due to B. Desjardins \cite[Lemma~2]{Desjardins} is one of the key tools for our paper. 
	\begin{lemma}
		\label{rhou}
		Suppose $\rho\geq 0$ and $u\in H^1(\mathbb{T}^2, \R^2)$ satisfy that $\sqrt{\rho}u\in L^2(\mathbb{T}^2, \R^2)$ and $\rho\in L^{\gamma}(\mathbb{T}^2)$ for some $\gamma>1$. Then, for any $q\in ]1,+\infty[$, there exists a constant $C=C(q,\gamma)$ such that the following estimate holds:
\begin{align*}
\left\|\rho^{\frac{1}{2q}}u \right\|_{L^{2q}}^q &\leq C\| \sqrt{\rho}u \|_{L^2}\| \nabla u \|_{L^2}^{q-1}\cdot \left\{ \log\left(2+\frac{\| \nabla u \|_{L^2}^2 \| \rho \|_{L^{\gamma}}}{\| \sqrt{\rho}u \|_{L^2}^2}\right)\right\}^{\frac{q-1}{2}}\\
 &\qquad +\| \sqrt{\rho}u \|_{L^2}\left|\int_{\mathbb{T}^2} u\,\mathrm{d} x\right|^{q-1}.   
\end{align*}
	\end{lemma}

	\section{A priori estimates for $\log \|\na u\|_{L^2}$}\label{sec: log Y}
This section is devoted to the derivation of \emph{a priori} estimates for the classical solution $(\rho, u)$ for the barotropic Navier--Stokes system~\eqref{equ} on $]0,\infty[ \times \tor$, whose existence is guaranteed by Theorem~\ref{thm: VK} (Va\u{\i}gant--Kazhikhov \cite{vaigant}). Our main result here is Proposition~\ref{propn: logY} below, in which we bound $\log \|\na u\|_{L^2}$ by some power of $\hatrho$, where $\hatrho(t) := \sup_{x \in \tor,\, 0 \leq \tau \leq t} \rho(\tau,x)$ as before.

It is crucial to note that all the estimates in the current section remain valid for \emph{weak} solutions when $\beta>3$ and $\gamma>1$ (Perepelitsa \cite[Section~6]{perep}). This observation is generalised to $\beta>3/2$ and $1<\gamma<4\beta-3$ (Huang--Li \cite{H1}), and can be justified in the general case $\beta>1$ and $\gamma>1$ by \S\ref{sec: final} in this paper. 

First of all, for any $t>0$ we have the conservation of mass and momentum, which follows direction from Equation~\eqref{equ}:
	\begin{equation}
		\label{mean}
		\int_{\mathbb{T}^2} \rho(t,x)\,\mathrm{d}x=\int_{\mathbb{T}^2} \rho_0(x)\,\mathrm{d}x
		\quad \text{and}\quad 
		\int_{\mathbb{T}^2} (\rho u)(t,x)\,\mathrm{d}x=\int_{\mathbb{T}^2} (\rho_0 u_0)(x)\,\mathrm{d}x.
	\end{equation}
 Without loss of generality, from now on assume that
 \begin{equation*}
     \int_{\mathbb{T}^2} \rho_0(x)\,\mathrm{d}x=1\quad\text{and}\quad \int_{\mathbb{T}^2} (\rho_0 u_0)(x)\,\mathrm{d}x=0.
 \end{equation*}

As in Serre \cite{serre1, serre2}, Hoff \cite{hoff1, hoff2, hoff-1d}, Lions \cite{lions}, and Feireisl--Novotn\'{y}--Petzeltova \cite{fnp}, etc., we introduce the important quantity $B$, known as the \emph{effective viscous flux}:	\begin{equation}
		\label{defB}
		B(t,x):=\big(\lambda(t,x)+2\mu \big)\mathrm{div} u(t,x) -\big(P(t,x)-\bar{P}(t)\big).
	\end{equation}
Clearly, by Equation~\eqref{equ} it solves the Poisson equation:
	$$\Delta B=\partial_t \big(\mathrm{div}(\rho u)\big)+\mathrm{div}\mathrm{div} (\rho u\otimes u).$$

 In addition, following Desjardins \cite{des2} we introduce 
	\begin{equation}
		\label{F}
		F:=2\mu \log\rho +\frac{\rho^{\beta}}{\beta}-(-\Delta)^{-1} \mathrm{div}(\rho u).
	\end{equation}
The conservation of momentum then becomes
	\begin{equation}
		\label{eq1}
		\begin{aligned}
			\frac{\mathrm{D}}{\mathrm{D} t}F +P-\bar{P} &= -\bar{B}-\bigg\{u\cdot (-\Delta)^{-\frac{1}{2}}\nabla(-\Delta)^{-\frac{1}{2}}\mathrm{div}(\rho u)\\
			&\qquad- (-\Delta)^{-\frac{1}{2}}\mathrm{div} (-\Delta)^{-\frac{1}{2}}\mathrm{div} [\rho u\otimes u]\bigg\}.
		\end{aligned}	
	\end{equation}
The terms in the parenthesis  on the right-hand side of Equation~\eqref{eq1}, thanks to the  important observation due to P.~L. Lions \cite{lions}, can be written as a sum of commutators of Riesz transforms and the multiplication operators by $u$:
	\begin{align}
		\label{ui}
		u\cdot (-\Delta)^{-\frac{1}{2}}\nabla(-\Delta)^{-\frac{1}{2}}\mathrm{div}(\rho u) &-(-\Delta)^{-\frac{1}{2}}\mathrm{div} (-\Delta)^{-\frac{1}{2}}\mathrm{div} [\rho u\otimes u]\nonumber\\
  &=\big[u_i,(-\Delta)^{-1}\partial_i\partial_j\big](\rho u_j) =: G.
	\end{align}
 Here and hereafter, Einstein's summation convention is adopted.

Following Perepelitsa \cite{perep} we introduce the following integral quantities:
 \begin{eqnarray}
    &&D^2(t) :=\int_{\mathbb{T}^2} \bigg\{(\lambda(t,x)+2\mu)|\mathrm{div} u(t,x)|^2+\mu|\mathrm{rot} u(t,x)|^2\bigg\} \,\mathrm{d}x,\label{D, def}\\
    && Y^2(t) :=\int_{\mathbb{T}^2} \left\{\mu|\mathrm{rot} u(t,x)|^2+\frac{B^2(t,x)}{\lambda(t,x)+2\mu}\right\}\,\mathrm{d}x,\label{Y, def}\\
&&X^2(t)=\int_{\mathbb{T}^2} \left\{\frac{\big|\nabla B(t,x)+\mu\nabla^{\bot} \mathrm{rot} u(t,x)\big|^2}{\rho(t,x)}\right\}\,\mathrm{d}x. \label{X, def}   
 \end{eqnarray}
 Denote also 
 \begin{equation*}
\tilde{\rho}(t):=\|\rho(t,\cdot)\|_{L^{\infty}}\quad \text{and} \quad \tilde{\rho_0}=\|\rho_0(\cdot)\|_{L^{\infty}},
 \end{equation*}
as well as 
\begin{equation*}
    \hatrho(t):= \sup_{\tau \in [0,t]} \trho(\tau).
\end{equation*}
In addition, from Equation~\eqref{mean} one directly infers that 
	\begin{equation}
		\label{suprho}
		\tilde{\rho}(t) \geq \int_{\mathbb{T}^2} \rho(t,\cdot) \,\mathrm{d}x=1.
	\end{equation}
This observation will be used throughout the subsequent developments.

The energy inequality is standard \cite{perep}. We sketch a proof for the sake of completeness.
\begin{lemma}[Energy inequality] 
\label{lemma: energy ineq}
There exists a uniform constant $C=C(\gamma)>0$ such that
\begin{equation*}
\begin{aligned}
&\sup\limits_{t>0} \int_{\mathbb{T}^2} \left\{\rho(t,x)\frac{|u(t,x)|^2}{2} +\frac{P(t,x)}{\gamma-1}\right\}\,\mathrm{d}x \\ &\qquad+\iint\limits_{]0,+\infty[\times \mathbb{T}^2} \bigg\{\big(\lambda(t,x)+2\mu\big)|\mathrm{div} u(t,x)|^2+\mu|\mathrm{rot} u(t,x)|^2\bigg\}\,\mathrm{d}x\,\mathrm{d}t\leq C\,E_0,
\end{aligned}
\end{equation*}
where
\begin{align*}
    E_0:=\int_{\mathbb{T}^2} \left\{\rho_0(x)\frac{|u_0(x)|^2}{2} +\frac{P\big(\rho_0(x)\big)}{\gamma-1}\right\}\,\mathrm{d}x.
\end{align*}
\end{lemma}
	\begin{proof}[Sketched proof of Lemma~\ref{lemma: energy ineq}]
		In view of $$\Delta u=\nabla \mathrm{div} u+\nabla^{\bot}\mathrm{rot} u,$$ one may rewrite the equation of conservation of momentum as 
		\begin{equation*}
			\rho \p_t u +\rho u\cdot\nabla u+\nabla P=\nabla\big(\lambda(\rho) +2\mu) \mathrm{div} u\big)+\mu \nabla^{\bot} \mathrm{rot} u.
		\end{equation*}
Hence, via integration by parts, we have that
		\begin{equation*}
			\frac{\mathrm{d}}{\mathrm{d}t}\int_{\mathbb{T}^2}  \rho\frac{|u|^2}{2}\,\dd x +
		\int_{\mathbb{T}^2} 	\bigg\{ \big((\lambda(\rho)+2\mu)\mathrm{div} u-P\big)\mathrm{div} u +\mu |\mathrm{rot} u|^2\bigg\}\,\mathrm{d}x=0.
		\end{equation*}

On the other hand, as $P=\rho^{\gamma}$ with $\gamma>1$, we deduce from the continuity equation that $$-(\gamma-1)P \mathrm{div} u=\p_t P +\mathrm{div}(Pu).$$ Thus
		\begin{equation}\label{x, June25}
			\frac{\mathrm{d}}{\mathrm{d}t}\int_{\mathbb{T}^2} \frac{P}{\gamma-1}\,\mathrm{d}x =-\int_{\mathbb{T}^2} P \mathrm{div} u\,\mathrm{d}x.
		\end{equation}
We conclude the proof from the previous two identities.  \end{proof}

Let us also observe the following bound:
\begin{equation}
\label{uu}  		\left|\int_{\mathbb{T}^2} u\,\mathrm{d}x\right|\leq C(\gamma, E_0)\|\nabla u\|_{L^2}. \end{equation}
Indeed, from H\"{o}lder and Sobolev inequalities and Lemma~\ref{lemma: energy ineq} we deduce that 
    $$\left|\int_{\mathbb{T}^2} (u-\bar{u})\rho\,\mathrm{d}x\right|\leq \|\rho\|_{L^{\gamma}}\|u-\bar{u}\|_{L^{\frac{\gamma}{\gamma-1}}}\leq C\|\nabla u\|_{L^2}.$$
On the other hand, thanks to the triangle inequality and Equation~\eqref{mean} (together with the ensuing normalisation condition), we have that
\begin{align*}
\left|\int_{\mathbb{T}^2}(u-\bar{u})\rho\,\mathrm{d}x \right| &\geq \left|\big|\bar{u}\big|\int_{\mathbb{T}^2}\rho_0\,\mathrm{d}x-\big|\int_{\mathbb{T}^2}\rho u\,\mathrm{d}x \big| \right|\\
&=\left|\big|\bar{u}\big|\int_{\mathbb{T}^2}\rho_0\,\mathrm{d}x-\big|\int_{\mathbb{T}^2}\rho_0 u_0\,\mathrm{d}x \big| \right| = \big|\bar{u}\big|.
\end{align*}
One thus obtains Equation~\eqref{uu}.

We shall need to bound $\|\na u\|_{L^4}$ and $\|\nabla u\|_{L^{2+\alpha}}$ in later developments (see, for instance, Proposition~\ref{propn: boundG}). A more general estimate is established below, in which we bound $\|\na u\|_{L^q}$ for any $q>2$ by powers of $\trho(t) \equiv \|\rho(t,\cdot)\|_\linf$, $D(t)$, and $X(t)$.

\begin{lemma}\label{lem: Lp for nabla u}
Let $q>2$ and $0<\varepsilon<1$ be arbitrary. Let $u$ be a strong solution to the Navier--Stokes Equations~\eqref{equ}. There exists a positive constant $C=C(\mu,\gamma, q,\varepsilon, E_0)$ such that
		\begin{equation*}
			\begin{aligned}
				\|\nabla u\|_{L^q}&\leq C \tilde{\rho} ^{\frac{\beta\varepsilon}{2}+\max\left\{0,\frac{\gamma-\beta}{2}\right\}\cdot\frac{2}{q}+\max\left\{0,\frac{\beta-\gamma}{2}\right\}\cdot(1-\frac{2}{q})}(t)\big[1+D(t)\big]\\
	  &\qquad+ C\tilde{\rho} ^{\frac{\beta\varepsilon}{2}+\frac{1}{2}-\frac{1}{q}+\max\left\{0,\frac{\gamma-\beta}{2}\right\}}(t)\big[1+D(t)\big]\left(\frac{X^2(t)}{10+Y^2(t)}\right)^{\frac{1}{2}-\frac{1}{q}}\\
      &\qquad + C \tilde{\rho}^{\max\left\{0,\frac{q-1}{q}\gamma-\beta\right\}}(t).			
			\end{aligned}
		\end{equation*}
  Here $D$ and $X$ are as in Equations~\eqref{D, def} and \eqref{X, def}.
	\end{lemma}

\begin{proof}[Proof of Lemma~\ref{lem: Lp for nabla u}]
By the div-curl estimate in Lemma~\ref{lemma1}, for any $q\in ]1,\infty[$ we have that
		$$\|\nabla B\|_{L^q} +\|\nabla \mathrm{rot} u\|_{L^q} \leq C(q,\mu) \left\|\nabla B +\mu\nabla^{\bot}\mathrm{rot} u\right\|_{L^q}.$$
		In particular, taking $q=2$ leads to 
  \begin{equation}\label{xx}
      \|\nabla B\|_{L^2} +\|\nabla \mathrm{rot} u\|_{L^2} \leq  C(\mu)\tilde{\rho}^{\frac{1}{2}}(t)X(t),
  \end{equation}
thanks to the definition of $X$ in \eqref{X, def}.

Now, by the triangle and Poincar\'{e} inequalities, it holds that
\begin{align*}
    \|B\|_{L^2} \leq \left|\bar{B}\right| + \left|B-\bar{B}\right| \leq \left|\bar{B}\right| + C\|\na B\|_{L^2},
\end{align*}
where $C$ is a universal constant. Hence,
\begin{align}
\label{BH1}    
\|B\|_{H^1} +\|\mathrm{rot} u\|_{H^1}\leq& C(\|\nabla B\|_{L^2} +\|\nabla \mathrm{rot} u\|_{L^2}) +\left|\int_{\mathbb{T}^2} B\,\mathrm{d}x\right| +\left|\int_{\mathbb{T}^2} \mathrm{rot} u\,\mathrm{d}x\right|\nonumber\\
\leq &C(\|\nabla B\|_{L^2} +\|\nabla \mathrm{rot} u\|_{L^2})+\left\{\int_{\mathbb{T}^2} (\lambda+2\mu )\,\mathrm{d}x\right\}^{\frac{1}{2}}\left\{\int_{\mathbb{T}^2} (\lambda+2\mu)|\mathrm{div}u|^2 \,\mathrm{d}x\right\}^{\frac{1}{2}}\nonumber\\
\leq &C\tilde{\rho}^{\frac{1}{2}}(t)X(t)+ \left(\|\rho\|_{L^{\beta}}^{\frac{\beta}{2}}+\sqrt{2\mu}\right)D(t)\nonumber\\
\leq &C(\gamma,\mu, E_0)\left\{\tilde{\rho}^{\frac{1}{2}}(t)X(t)+ \tilde{\rho}^{\max\left\{0,\frac{\beta-\gamma}{2}\right\}}(t)D(t)\right\}.
\end{align} 
In the second line we used the definition of $B$, the identities $\int_{\tor}{\rm rot} u\,\dd x = \int_\tor \left(P-\bar{P}\right)\,\dd x =0$, and the Cauchy--Schwarz inequality; in the third line we used the bound~\eqref{xx}, the definition of $D$ in \eqref{D, def}, and the assumption $\lambda = \rho^\beta$; and, in the final line, we invoked the bound $\int_\tor \rho^\gamma\,\dd x \leq C(\gamma, E_0)$ via Lemma~\ref{lemma: energy ineq} and $\trho \geq 1$ as in \eqref{suprho}.

Now we are ready to estimate $\|\na u\|_{L^q}$. By the div-curl estimate in Lemma~\ref{lemma1}, we have
\begin{align*}
    \|\nabla u\|_{L^q}\leq C(q)\left(\|\mathrm{div} u\|_{L^q}+\|\mathrm{rot} u\|_{L^q}\right)
\end{align*}
for any $q>1$. To control the right-hand side, we note by the definition of $B$ in \eqref{defB} that 
\begin{align*}
   \|\mathrm{div} u\|_{L^q} \leq  \left\|\frac{B}{\lambda+2\mu}\right\|_{L^q} +\left\|\frac{P-\bar{P}}{\lambda+2\mu}\right\|_{L^q}
\end{align*}
The right-hand side is controlled as follows:
\begin{itemize}
    \item 
    The first term $ \left\|\frac{B}{\lambda+2\mu}\right\|_{L^q} $ is controlled via the H\"{o}lder inequality, the definition of $Y$ in \eqref{Y, def}, and the Gagliardo--Nirenberg--Sobolev interpolation inequality as follows, for any $q>2$ and $\e>0$:
\begin{align*}
 \left\|\frac{B}{\lambda+2\mu}\right\|_{L^q}   &\leq \left\|\frac{B}{\lambda+2\mu}\right\|_{L^2}^{\frac{2}{q}-\varepsilon}\left\|B\right\|_{L^{\frac{2(1+\varepsilon)q-4}{q\varepsilon}}}^{1-\frac{2}{q}+\varepsilon}\\
 &\leq Y^{\frac{2}{q}-\varepsilon}(t)\|B\|_{L^2}^{\varepsilon}\|B\|_{H^1}^{1-\frac{2}{q}}.
\end{align*}
\item 
For the second term $\left\|\frac{P-\bar{P}}{\lambda+2\mu}\right\|_{L^q}$, let us bound for any $q>1$ that
\begin{align}\label{NOV24, to use}
   \left\|\frac{P-\bar{P}}{\lambda+2\mu}\right\|_{L^q} &\leq C \left\{\int_\tor \rho^{(\gamma-\beta)q}\,\dd x\right\}^{\frac{1}{q}}\nonumber\\
   &\leq C\left\{\int_\tor \rho^\gamma\,\dd x\right\}^{\frac{1}{q}} \tilde{\rho}^{\max\left\{0,\frac{q-1}{q}\gamma-\beta\right\}}(t)\nonumber\\
   &\leq C(q, E_0)\tilde{\rho}^{\max\left\{0,\frac{q-1}{q}\gamma-\beta\right\}}(t),
\end{align}
where we use the energy inequality in Lemma~\ref{lemma: energy ineq} to control $\int_\tor \rho^\gamma\,\dd x \leq C(\gamma, E_0)$. 
\end{itemize}

The estimate for ${\rm rot}u$ is simpler: we have
\begin{align*}
    \|{\rm rot} u\|_{L^q} &\leq \|\mathrm{rot} u\|_{L^2}^{\frac{2}{q}}\|\mathrm{rot} u\|_{H^1}^{1-\frac{2}{q}} \\
    &\leq Y^{\frac{2}{q}}(t) \|\mathrm{rot} u\|_{H^1}^{1-\frac{2}{q}},
\end{align*}
thanks to the Gagliardo--Nirenberg--Sobolev inequality and the definition of $Y$ in \eqref{Y, def}.

Putting together the above estimates for ${\rm div} u$ and ${\rm rot} u$, we arrive at
\begin{align*}
    \|\na u\|_{L^q} \leq C(q,\varepsilon,E_0) \left\{ Y^{\frac{2}{q}-\varepsilon}(t)\|B\|_{L^2}^{\varepsilon}\|B\|_{H^1}^{1-\frac{2}{q}} +Y^{\frac{2}{q}}(t)\|\mathrm{rot} u\|_{H^1}^{1-\frac{2}{q}} +\tilde{\rho}^{\max\left\{0,\frac{q-1}{q}\gamma-\beta\right\}}(t)\right\}.
\end{align*}
In addition, by the definition of $Y$ in \eqref{Y, def} and the assumption~\eqref{assumption}, one has  $$\frac{\|B\|_{L^2}}{Y} \leq \sqrt{\|\lambda\|_{\linf}} \leq \tilde{\rho}^{\frac{\beta}{2}}(t).$$ Thus
\begin{align}\label{yy}
    \|\na u\|_{L^q} \leq C(q,\varepsilon,E_0)\left\{ \tilde{\rho} ^{\frac{\beta\varepsilon}{2}}(t)Y^{\frac{2}{q}} \left( \|B\|_{H^1} +\|\mathrm{rot} u\|_{H^1} \right)^{1-\frac{2}{q}} +\tilde{\rho}^{\max\left\{0,\frac{q-1}{q}\gamma-\beta\right\}}(t)\right\}.
\end{align}

To proceed, by virtue of the definitions of $Y$ and $D$ in \eqref{Y, def} and \eqref{D, def} respectively, Young's inequality, the assumption~\eqref{assumption}, and the energy inequality in Lemma~\ref{lemma: energy ineq}, observe the following:
\begin{align*}
    Y^2(t)-D^2(t) &= \int_\tor \left\{ \frac{B^2}{\lambda+2\mu} - (\lambda+2\mu)|{\rm div} u|^2 \right\}\,\dd x\\
    &\leq \int_\tor \left\{(\lambda+2\mu)|{\rm div} u|^2  + 2\frac{\left|P-\bar{P}\right|^2}{\lambda+2\mu}\right\}\,\dd x\\
    &\leq D^2(t) + C\int_\tor \rho^{2\gamma-\beta}\,\dd x\\
    &\leq D^2(t) + C(E_0) \cdot\tilde{\rho}^{\max\left\{0,\gamma-\beta\right\}}(t).
\end{align*}
This implies that
\begin{equation}
\label{DY2}
	Y(t)\leq C(E_0) \cdot \left(\tilde{\rho}^{\max\left\{0,\frac{\gamma-\beta}{2}\right\}}(t)+D(t)\right).
 \end{equation}

To conclude, putting together the estimates in \eqref{BH1}, \eqref{yy}, and \eqref{DY2},  we obtain that
\begin{align*}
      \|\na u\|_{L^q} &\leq C\trho^{\frac{\beta\e}{2}}(t) \cdot\trho^{\max\left\{0,\frac{\gamma-\beta}{2}\right\}\cdot{\frac{2}{q}}}(t) Y^{\frac{2}{q}}\left\{ \trho^{\frac{1}{2}}(t)X + \left(1+\trho^{\max\left\{0, \frac{\beta-\gamma}{2}\right\}}\right)D \right\}^{1-\frac{2}{q}} \nonumber\\
	&\qquad + C \trho^{\max\left\{0,\frac{q-1}{q}\gamma-\beta\right\}}(t)\\
		&\leq C \tilde{\rho} ^{\frac{\beta\varepsilon}{2}+\max\left\{0,\frac{\gamma-\beta}{2}\right\}\cdot\frac{2}{q}+\max\left\{0,\frac{\beta-\gamma}{2}\right\}\cdot(1-\frac{2}{q})}(t)\big[1+D(t)\big]\\
	  &\qquad+ C\tilde{\rho} ^{\frac{\beta\varepsilon}{2}+\frac{1}{2}-\frac{1}{q}+\max\left\{0,\frac{\gamma-\beta}{2}\right\}}(t)\big[1+D(t)\big]\left(\frac{X^2}{10+Y^2}\right)^{\frac{1}{2}-\frac{1}{q}}\\
      &\qquad + C \tilde{\rho}^{\max\left\{0,\frac{q-1}{q}\gamma-\beta\right\}}(t)
 \end{align*}
with constants $C=C(\gamma,\mu,\e,q,E_0)$. From here, the desired inequality can be deduced by noting that $(a+b)^{1-\frac{2}{q}} \leq a^{1-\frac{2}{q}}+b^{1-\frac{2}{q}}$ for $q>2$; $a,b \geq 0$, and that $\trho \geq 1$ by \eqref{suprho}.  \end{proof}

\begin{remark}
    \label{nablau, L2}
    In particular, for $q=2$ we have that
    \begin{equation*}
     \begin{aligned}
        \|\nabla u\|_{L^2}&\leq \|\mathrm{div}u\|_{L^2}+\|\mathrm{rot}u\|_{L^2}\\
        &\leq C\left(\left\|\frac{B}{\lambda+2\mu}\right\|_{L^2}+\left\|\frac{P-\bar{P}}{\lambda+2\mu}\right\|_{L^2}+\|\mathrm{rot}u\|_{L^2}\right)\\
        &\leq C(E_0) \left(Y(t)+\tilde{\rho}^{\max\left\{0,\frac{1}{2}\gamma-\beta \right\}}(t)\right).
     \end{aligned}
    \end{equation*}
Here, we used the definition of $Y$ in \eqref{Y, def} and the case $q=2$ in \eqref{NOV24, to use}.
\end{remark}


Now, we are in the situation of introducing the main estimate of this section.
 
\begin{proposition}
\label{propn: logY}
Let $0<\varepsilon<1$ and $\beta>1$ be arbitrary. Let $u$ be a strong solution to the Navier--Stokes Equations~\eqref{equ}. There exists a constant $C=C(\varepsilon, M, \mu, \beta, \gamma, E_0)$ such that
\begin{equation*}
    \log\Big(10 + \|\na u\|^2_{L^2} \Big)\leq C\hatrho^{\,\varsigma},
\end{equation*}
where
\begin{align*}
    \varsigma:={1+\beta\varepsilon+\max\left\{0, \gamma-2\beta,\beta-\gamma-2\right\}}.
\end{align*}
  Recall here that $\hatrho(t) := \sup_{[0,t] \times \tor}\rho$ and $M$ is the essential supremum of the initial density $\rho_0$.
\end{proposition}

\begin{proof}[Proof of Proposition~\ref{propn: logY}]
We divide our arguments into five steps below.

The first four steps are devoted to establishing a bound for $\log Y$ of independent interest:
\begin{equation}\label{log Y bound}
\log\big(10+Y^2(t)\big)+\int_{0}^{t} \frac{X^2(\tau)}{10+Y^2(\tau)}\,\mathrm{d}\tau\leq C\hatrho^{\,\varsigma}(t).
\end{equation}
This is motivated by Perepelitsa \cite[p.1141, Corollary~2]{perep}; nonetheless, our proof does not involve the Orlicz space techniques as in \cite{perep}, and our result differs from \cite{perep} by providing a bound in terms of a power of $\hatrho$ only. Finally, in Step~5 we deduce the desired bound for $\log\big(10 + \|\na u\|^2_{L^2} \big)$.

\smallskip
\noindent
{\bf Step 1.} By utilising the identity $$\Delta u=\nabla \mathrm{div} u+\nabla^{\bot}\mathrm{rot} u,$$ we recast the equation for the conservation of momentum into the following form: 
\begin{equation}
			\label{du}
			\frac{\mathrm{D} u}{\mathrm{D}t}+\frac{1}{\rho}\nabla\big(P-\bar{P}-(\lambda+2\mu)\mathrm{div} u\big)-\frac{1}{\rho}\mu\nabla^{\bot}\mathrm{rot} u=0.
\end{equation}
Next, applying $\mathrm{rot}$ and $\mathrm{div}$ to Equation~\eqref{du} above and noting the definition of $B$ in \eqref{defB}, we obtain the two scalar equations below:
\begin{align}\label{drotu}
\frac{\mathrm{D}}{\mathrm{D}t}\mathrm{rot} u-\mathrm{rot} u \,\mathrm{div} u &=\frac{\mathrm{D}}{\mathrm{D}t}\mathrm{rot} u -(\partial_1 u\cdot\nabla)u_2 +(\partial_2 u\cdot \nabla)u_1 \nonumber\\
&=\mathrm{rot}\frac{1}{\rho}\left(\nabla B+\mu \nabla^{\bot}\mathrm{rot} u\right)  
\end{align}
and
\begin{align}\label{ddivu}
&  \frac{\mathrm{D}}{\mathrm{D}t} \mathrm{div} u +(\partial_1 u\cdot\nabla)u_1 +(\partial_2 u\cdot\nabla)u_2\nonumber\\&\qquad\qquad=\frac{\mathrm{D}}{\mathrm{D}t}\left(\frac{B}{\lambda+2\mu}\right)+\frac{\mathrm{D} }{\mathrm{D}t}\left(\frac{P-\bar{P}}{\lambda+2\mu}\right)+|\mathrm{div} u|^2- 2\nabla u_1\cdot \nabla^{\bot}u_2\nonumber\\
&\qquad\qquad=\mathrm{div}\frac{1}{\rho}(\nabla B+\mu\nabla^{\bot}\mathrm{rot} u).  
\end{align}

\smallskip
\noindent
{\bf Step 2.} We now derive an energy identity, Equation~\eqref{sum} below, for (suitably weighted) $L^2$-norms of $\rotu$ and $B$. We do not claim any novelty here: our identity~\eqref{sum} is the same as \cite[Equation~(39)]{perep}.

To this end, we first multiply Equation~\eqref{drotu} by $2\mu\,\mathrm{rot}u$ to obtain that
		\begin{equation}
			\label{rotu2}
			\mu\frac{\mathrm{D}}{\mathrm{D}t}|\mathrm{rot} u|^2 +2\mu\,|\mathrm{rot} u|^2 \mathrm{div} u =2\mu\,\mathrm{rot} u \, \mathrm{rot} \left[\frac{1}{\rho}(\nabla B+\mu \nabla^{\bot} \mathrm{rot} u)\right].
		\end{equation}
We also obtain via multiplying Eq.~\eqref{ddivu} by $2B$ that 
\begin{align}\label{B2}
&\frac{\mathrm{D}}{\mathrm{D}t} \left(\frac{B^2}{\lambda+2\mu}\right)+\frac{B^2}{\lambda+2\mu}\mathrm{div}u-\frac{B^2}{\lambda+2\mu}\mathrm{div}u +B^2\frac{\mathrm{D}}{\mathrm{D}t}\left(\frac{1}{\lambda+2\mu}\right)\nonumber\\
&\qquad+2B\frac{\mathrm{D}}{\mathrm{D}t}\left(\frac{P-\bar{P}}{\lambda+2\mu}\right) +2B|\mathrm{div}u|^2-4B\nabla u_1\cdot \nabla^{\bot}u_2 \nonumber\\
&\qquad\qquad =  2B\mathrm{div}\left[\frac{1}{\rho} (\nabla B+\mu\nabla^{\bot}\mathrm{rot}u)\right].
\end{align}
In Equation~\eqref{B2} above, one may further express
\begin{align*}
      \frac{\mathrm{D}}{\mathrm{D}t}\left(\frac{P-\bar{P}}{\lambda+2\mu}\right)&=\beta \frac{\lambda(P-\bar{P})}{\lambda+2\mu}\mathrm{div}u-\frac{\gamma P}{\lambda+2\mu}\mathrm{div}u \nonumber\\
      &\qquad+\frac{\gamma-1}{\lambda+2\mu}\int_{\mathbb{T}^2} (P-\bar{P})\mathrm{div}u \,\mathrm{d}x,
\end{align*}
thanks to the continuity equation~$\eqref{PDE, 1}_1$. Also, for $\lambda=\lambda(\rho)=\rho^\beta$ as in assumption~\eqref{assumption}, one has 
\begin{align*}
        \frac{\mathrm{D}}{\mathrm{D}t}\left(\frac{1}{\lambda+2\mu}\right) &=-\frac{\lambda_t+u\cdot\nabla \lambda}{(\lambda+2\mu)^2} \nonumber\\
        &=\frac{\beta\lambda \mathrm{div}u}{(\lambda+2\mu)^2}.
       \end{align*}
In view of the two identities above, we add up together Equations~\eqref{rotu2} and \eqref{B2} and integrate over $\mathbb{T}^2$ to deduce that
   \begin{align}
        \label{sum}
		&\frac{\mathrm{d}}{\mathrm{d}t}\left[ \int_{\mathbb{T}^2} \mu|\mathrm{rot} u|^2 +\frac{B^2}{\lambda+2\mu} \,\mathrm{d}x\right] +2\int_{\mathbb{T}^2} \frac{\left|\nabla B+\mu \nabla^{\bot} \mathrm{rot} u\right|^2}{\rho}\,\mathrm{d}x \nonumber\\
		=&-\mu\int_{\mathbb{T}^2} |\mathrm{rot} u|^2 \mathrm{div} u \,\mathrm{d}x +4\int_{\mathbb{T}^2} B\nabla u_1\cdot \nabla^{\bot}u_2\,\mathrm{d}x -2\int_{\mathbb{T}^2} B|\mathrm{div} u|^2\,\mathrm{d}x \nonumber\\
		&-\int_{\mathbb{T}^2} \frac{(\beta-1)\lambda-2\mu}{(\lambda+2\mu)^2}B^2 \mathrm{div} u\,\mathrm{d}x -2\beta\int_{\mathbb{T}^2} \frac{\lambda(P-\bar{P})}{(\lambda+2\mu)^2}B \mathrm{div} u\,\mathrm{d}x \nonumber\\
		&+2\gamma\int_{\mathbb{T}^2} \frac{P}{\lambda+2\mu} B\mathrm{div}u\,\mathrm{d}x-2(\gamma-1)\left\{\int_{\mathbb{T}^2} (P-\bar{P})\mathrm{div}u \,\mathrm{d}x\right\}\left\{\int_{\mathbb{T}^2} \frac{B}{\lambda+2\mu} \,\mathrm{d}x\right\} \nonumber\\
		=:&\sum_{i=1}^{7}I_i.
    \end{align}

\smallskip
\noindent
{\bf Step 3.} In this step, we shall estimate each of the terms $I_1, \ldots, I_7$.

\begin{enumerate}
    \item 
The estimate for $I_1$ is straightforward. In view of the H\"{o}lder and Sobolev inequalities and the definition of $D$ in \eqref{D, def}, we have that
\begin{align}
 \label{I1}
I_1 &\leq C\|\mathrm{rot} u\|_{L^4}^2\|\mathrm{div}u\|_{L^2}\nonumber\\
&\leq C\|\mathrm{rot}u\|_{L^2}\|\mathrm{rot} u\|_{H^1}\|\mathrm{div}u\|_{L^2}\nonumber\\
&\leq C(E_0) D(t) \|\mathrm{rot}u\|_{H^1} \left(Y(t)+\tilde{\rho}^{\max\left\{0,\frac{1}{2}\gamma-\beta \right\}}\right),
\end{align}
where we use the fact that $\|\mathrm{div}u\|_{L^2}\leq \|\nabla u\|_{L^2}$ and Remark \ref{nablau, L2} to control $\|\nabla u\|_{L^2}$.
\item 
The estimate for $I_2$ is classical in the literature of PDEs for compressible fluid mechanics \cite{lions, fnp}. Indeed, one may bound
\begin{align}
    \label{I2}
I_2 &\leq C\|B\|_{\mathcal{BMO}} \left\|\nabla u_1\cdot\nabla^{\bot} u_2\right\|_{\mathcal{H}^1} \nonumber\\
&\leq C\|\nabla B\|_{L^2}\|\nabla u\|_{L^2}^2\nonumber\\
&\leq C(E_0)\|B\|_{\dot{H}^1} D(t)\left(Y(t)+\tilde{\rho}^{\max\left\{0,\frac{1}{2}\gamma-\beta \right\}}\right).
\end{align}
The first line follows from the Hardy-BMO duality; the second line holds by the continuous embedding $\dot{H}^1(\tor) \emb \mathcal{BMO}(\tor)$ and the div-curl lemma of the compensated compactness theory (note that $\mathrm{rot} \nabla u_1= \mathrm{div} \nabla^{\bot} u_2=0$; see Coifman--Lions--Meyer--Semmes \cite[Theorem II.1]{Coi}); and the final line follows from the definition of $D$ in \eqref{D, def} and Remark \ref{nablau, L2}. 

\item 

For $I_3$, we estimate 
\begin{align}\label{I3}
|I_3| &\leq 2 \int_\tor |B| \left|\frac{B+P-\bar{P}}{\lambda+2\mu}\right||\divu|\,\dd x\nonumber\\
&\leq C\int_{\mathbb{T}^2} \frac{B^2|\mathrm{div}u|}{\lambda+2\mu}\,\mathrm{d}x+C\int_{\mathbb{T}^2}\frac{|B|\left|P-\bar{P}\right||\mathrm{div} u|}{\lambda+2\mu}\,\mathrm{d}x\nonumber\\
&\leq C\left\|\frac{B^2}{\lambda+2\mu}\right\|_{L^2}\|\mathrm{div} u\|_{L^2}+C\left\|\frac{P-\bar{P}}{(\lambda+2\mu)^{\frac{3}{2}}}\right\|_{L^{2+\frac{2\beta}{\gamma}}}\|B\|_{L^{2+\frac{2\gamma}{\beta}}}\|\sqrt{\lambda+2\mu}\,\mathrm{div}u\|_{L^2}\nonumber\\
&=: I_{3,1}+I_{3,2}.
\end{align}
Here the first line utilises the definition of $B$ in \eqref{defB} and follows Perepelitsa \cite[Equation~(44)]{perep}, and the third line holds by the H\"{o}lder inequality.
\begin{itemize}
    \item 
For $I_{3,1}$, we apply the crude bound $\|\divu\|_{L^2} \leq C(\mu)D$ and the following estimate via H\"{o}lder's inequality, the definition of $Y$ in \eqref{Y, def}, and Lemma~\ref{1}:
\begin{align*}
\left\|\frac{B^2}{\lambda+2\mu}\right\|_{L^2} &\leq C(\mu,\e) \left\{\int_\tor\frac{B^{2(1-\e)}}{(\lambda+2\mu)^{1-\e}} \cdot B^{2(1+\e)}\,\dd x\right\}^{\frac{1}{2}} \\
&\leq C(\mu,\e) \left\|\frac{B}{\sqrt{\lambda+2\mu}}\right\|_{L^2}^{1-\varepsilon}\|B\|_{L^{\frac{2(1+\e)}{\varepsilon}}}^{1+\varepsilon} \\
&\leq C(\mu,\e) Y^{1-\e}\|B\|_{L^2}^\e \|B\|_{H^1}.
\end{align*}
Hence, $I_{3,1} \leq C(\mu,\e)DY^{1-\e}\|B\|_{L^2}^\e \|B\|_{H^1}$. In addition, by the definition of $Y$ in \eqref{Y, def}, the assumptions in \eqref{assumption} and that $\trho \geq 1$ in \eqref{suprho}, we further have that
\begin{align*}
    \|B\|_{L^2}^\e &= \left\{\int_\tor \frac{B^2}{\lambda+2\mu} \cdot (\lambda+2\mu) \,\dd x \right\}^{\frac{\e}{2}}\\
    &\leq C(\mu) Y^\e \trho^{\frac{\beta \e}{2}}.
\end{align*}
We thus arrive at
\begin{align*}
    I_{3,1} \leq C(\mu,\e) \trho^{\frac{\beta\e}{2}} DY \|B\|_{H^1}. 
\end{align*}
    \item 
    For $I_{3,2}$, we use the energy estimates to obtain
    \begin{align}\label{NOV24, to use2}
      \left\|\frac{P-\bar{P}}{(\lambda+2\mu)^{\frac{3}{2}}}\right\|_{L^{2+\frac{2\beta}{\gamma}}}&\leq C \left\{\int_\tor \rho^{(\gamma-\frac{3}{2}\beta)(2+\frac{2\beta}{\gamma})}\,\dd x\right\}^{\frac{\gamma}{2\gamma+2\beta}}\nonumber\\
    &\leq C\left\{\int_\tor \rho^\gamma\,\dd x\right\}^{\frac{\gamma}{2\gamma+2\beta}} \tilde{\rho}^{\max\left\{0,\frac{\gamma+2\beta}{2\gamma+2\beta}\gamma-\frac{3}{2}\beta\right\}}(t)\nonumber\\
    &\leq C(q, E_0)\tilde{\rho}^{\max\left\{0,\frac{1}{2}\gamma-\beta\right\}}(t),
    \end{align}
    where the third line we used the fact that $$\frac{\gamma+2\beta}{2\gamma+2\beta}\gamma-\frac{3}{2}\beta=\frac{1}{2}\gamma-\beta+\frac{\beta}{2\gamma+2\beta}\gamma-\frac{1}{2}\beta<\frac{1}{2}\gamma-\beta.$$
 In view of the definition of $D$ in \eqref{D, def}, it is clear that
\begin{align*}
   I_{3,2} \leq C(\mu) \tilde{\rho}^{\max\left\{0,\frac{1}{2}\gamma-\beta\right\}}D\|B\|_{H^1}.
\end{align*}
\end{itemize}

In summary, we have
\begin{align}
    \label{I3, final}
    |I_3| \leq C(\mu,\e) \trho^{\frac{\beta\e}{2}} DY \|B\|_{H^1} +   C(\mu) \tilde{\rho}^{\max\left\{0,\frac{1}{2}\gamma-\beta\right\}}D\|B\|_{H^1}.
\end{align}

\item 
Next, for $I_4$, $I_5$, and $I_6$, observe that they are all bounded by the expression in the second or the third line of \eqref{I3} above. Indeed, for $I_4$ we simply note that $\left|\frac{(\beta-1)\lambda-2\mu}{(\lambda+2\mu)^2}\right| \leq \beta$, so 
\begin{align*}
|I_4| \leq \beta \int_{\mathbb{T}^2} \frac{B^2|\mathrm{div}u|}{\lambda+2\mu}\,\mathrm{d}x.
\end{align*}
For $I_5$, it is clear that
\begin{align*}
|I_5| \leq 2\beta    \int_{\mathbb{T}^2}\frac{|B|\left|P-\bar{P}\right||\mathrm{div} u|}{\lambda+2\mu}\,\mathrm{d}x.
\end{align*}
Then, for $I_6$, since $\frac{\bar{P}}{\lambda+2\mu}\leq C(\mu,E_0)$, by triangle and H\"{o}lder inequalities, we have that
\begin{align*}
    |I_6| &\leq 2\gamma \int_{\mathbb{T}^2}\frac{|B|\left|P-\bar{P}\right||\mathrm{div} u|}{\lambda+2\mu}\,\mathrm{d}x+2\gamma \int_{\mathbb{T}^2}\frac{|B||\bar{P}||\mathrm{div} u|}{\lambda+2\mu}\,\mathrm{d}x\\
    &\leq C\int_{\mathbb{T}^2}\frac{|B|\left|P-\bar{P}\right||\mathrm{div} u|}{\lambda+2\mu}\,\mathrm{d}x +C\|B\|_{L^2}\|\mathrm{div}u\|_{L^2}.
\end{align*}
Comparing the above bounds for $I_4$, $I_5$, and $I_6$ with the second and the third lines in \eqref{I3} above, we deduce that
\begin{align}\label{I4 -- I6}
|I_4|+|I_5|+|I_6| &\leq CD(t)\|B\|_{H^1} \left\{\trho^{\frac{\beta\e}{2}}(t)Y(t) + \trho^{\max\{0,\frac{1}{2}\gamma-\beta\}}(t)\right\},
\end{align}
with some constant $C=C(\e,\mu,\beta,\gamma)$. Here $\e \in ]0,1[$ is arbitrary.

\item 

Finally let us turn to $I_7$. Note by $\int_{\tor}\mathrm{div}u\,\mathrm{d}x=0$ and the energy inequality (Lemma~\ref{lemma: energy ineq}) that
\begin{align*}
    \left|\int_\tor \frac{B}{\lambda+2\mu}\,\dd x\right| = \left|\int_{\tor} \mathrm{div}u-\frac{P-\bar{P}}{\lambda+2\mu}\,\mathrm{d}x \right|\leq C(\mu)\int_{\tor}|P-\bar{P}|\,\mathrm{d}x \leq C(\mu,E_0).
\end{align*}
On the other hand, from the definition of $B$ in \eqref{defB}, we have that
\begin{align*}
&\left|\int_{\mathbb{T}^2} (P-\bar{P})\mathrm{div}u \,\mathrm{d}x\right| = \left|\int_{\mathbb{T}^2} \left((\lambda+2\mu)\mathrm{div}u-B\right)\mathrm{div}u \,\mathrm{d}x\right|\\
&\qquad \leq D^2(t)+\|B\|_{L^2}\|\mathrm{div}u\|_{L^2}\\
&\qquad \leq C\Big\{D^2(t)+\|B\|_{H^1}D(t)\Big\},
\end{align*}
thanks to the Cauchy--Schwarz inequality and the definition of $D$ in \eqref{D, def}. Hence,
\begin{equation}\label{I7}
    |I_7| \leq C(E_0) \|B\|_{H^1}\Big\{D(t)+C(E_0)D^2(t)\Big\}.
\end{equation}

\end{enumerate}

\smallskip
\noindent
{\bf Step 4.} Now we are ready to conclude the estimate~\eqref{log Y bound}. Substituting the estimates~\eqref{I1}, \eqref{I2}, \eqref{I3, final}, \eqref{I4 -- I6}, and \eqref{I7} into the identity~\eqref{sum}, noting by \eqref{suprho} that $\trho \geq 1$, and making use of the bounds for $\|B\|_{H^1}$ and $\|\rotu\|_{H^1}$ in \eqref{BH1}, we deduce the differential inequality: 
\begin{align*}
    \frac{\mathrm{d} Y^2}{\mathrm{d}t} +2X^2 &\leq C \cdot D \left[\trho^{\frac{1}{2}}X + \trho^{\max\left\{\frac{\beta-\gamma}{2},0\right\}}D\right] \left[ \trho^{\frac{\beta\e}{2}}Y + \trho^{\max\{\frac{1}{2}\gamma-\beta,0\} }\right].
\end{align*}
where $0<\e<1$ is arbitrary, and $C=C(\e,E_0,\mu,\beta,\gamma)$. By Young's inequality it holds that
\begin{align*}
			\frac{\mathrm{d} Y^2}{\mathrm{d}t} +  X^2 &\leq C\trho D^2   \left[\trho^{\frac{\beta\e}{2}}Y + \trho^{\max\{\frac{1}{2}\gamma-\beta,0\} }\right]^2 + C\trho^{\max\left\{\frac{\beta-\gamma}{2},0\right\}-\frac{1}{2}} D \cdot \sqrt{\trho} D \left[\trho^{\frac{\beta\e}{2}}Y + \trho^{\max\{\frac{1}{2}\gamma-\beta,0\} }\right]\\
			&\leq C'\trho D^2   \left[\trho^{\frac{\beta\e}{2}}Y + \trho^{\max\{\frac{1}{2}\gamma-\beta,0\} }\right]^2 + C'D^2 \trho^{\max\{\beta-\gamma,0\}-1},
		\end{align*}
		for $C$ and $C'$ depending only on $\e$, $E_0$, $\mu$, $\beta$, and $\gamma$ as before. We thus arrive at
		\begin{equation}
			\label{dY}
			\begin{aligned}
				\frac{\mathrm{d} Y^2}{\mathrm{d}t} +X^2\leq CD^2\tilde{\rho}\left\{\tilde{\rho}^{\beta\varepsilon}Y^2 +\tilde{\rho}^{\max\left\{0, \gamma-2\beta \right\}} +\tilde{\rho}^{\max\left\{0, \beta-\gamma\right\}-2}\right\},
			\end{aligned}
		\end{equation}
		where  $C=C(\e, E_0, \mu,\beta,\gamma)$. 
Dividing by $10 + Y^2$ on both sides of the estimate above and integrating over $]0,t[$, one infers that
\begin{align*}
&\log(10+Y^2(t))+\int_{0}^{t} \frac{X^2(\tau)}{10+Y^2(\tau)}\,\mathrm{d}\tau \nonumber\\
&\qquad \leq \log(10+Y^2(0))+ C \int_0^t D^2(\tau) \tilde{\rho}^{\,{1+\beta\varepsilon+\max\left\{0,  \gamma-2\beta,\beta-\gamma-2\right\}}}(\tau)\,\dd\tau.
\end{align*}
In view of Lemma~\ref{lemma: energy ineq}, it holds that $\int_0^\infty D^2(t)\,\dd t\leq C(\gamma,E_0)$. Thus, 
denoting
\begin{equation*}
    \hatrho(t):= \sup_{\tau \in [0,t]} \trho(\tau) \equiv  \|\rho(\cdot,\cdot)\|_{\linf([0,t]\times \tor)},
\end{equation*}
we conclude that
\begin{equation*}
\log\left(10+Y^2(t)\right)+\int_{0}^{t} \frac{X^2(\tau)}{10+Y^2(\tau)}\,\mathrm{d}\tau\leq C{\hatrho}^{\,1+\beta\varepsilon+\max\left\{0,  \gamma-2\beta,\beta-\gamma-2\right\}}(t).
\end{equation*}
Here  $C=C(\e, E_0, Y(0),  \mu,\beta,\gamma)$, where $Y(0)$ is completely determined by $M$ (the essential upper bound for $\rho_0$), $\mu$, $\gamma$, and $E_0$.  This completes the proof of Equation~\eqref{log Y bound}.

\smallskip
\noindent
{\bf Step 5.} Finally, to deduce the bound for $\log\big(10+\|\nabla u\|_{L^2}^2(t)\big)$, let us first note by the div-curl estimate in Lemma~\ref{lemma1} that 
\begin{align*}
    \log\left(10+\|\nabla u\|_{L^2}^2\right) \leq C\log\left\{10 + \int_\tor \left(|\rotu|^2+|\divu|^2\right)\,\dd x\right\}.
\end{align*}
In view of the definition of $B$ and $Y$ in Equations~\eqref{defB} and \eqref{Y, def} and the assumptions for $P$, $\lambda$, and $\mu$ in \eqref{assumption}, we have that
\begin{align*}
 \int_\tor \left(|\rotu|^2+|\divu|^2\right)\,\dd x &\leq \int_\tor \left\{|\rotu|^2 + \frac{B^2}{(\lambda+2\mu)^2}+\frac{(P-\bar{P})^2}{(\lambda+2\mu)^2}\right\}\,\dd x\\
 &\leq C(\mu) \left(Y(t)^2 +\tilde{\rho}^{\max\left\{0, \gamma-2\beta\right\}}(t)\right).
\end{align*}
Thus, from Equation~\eqref{log Y bound} we deduce that
\begin{align*}
    \log\left(10+\|\nabla u\|_{L^2}^2\right) \leq C(\mu)\hatrho^{\,1+\beta\varepsilon+\max\left\{0, \gamma-2\beta,\beta-\gamma-2\right\}}.
\end{align*}
This completes the proof.  \end{proof}


 \section{$\linf$-bound for the commutator term $G$}\label{sec: G}

In this section, we estimate the $\linf$-norm of the commutator term $G$, which turns out to be crucial for showing that neither vacuum nor concentration forms in the fluid at $t = +\infty$. Recall from Equation~\eqref{G, def} that
\begin{align*}
G := \sum_{i,j=1}^2 \left[u_i,(-\Delta)^{-1}\partial_i\partial_j\right](\rho u_j).
\end{align*}
The key novel point of our arguments is to apply Desjardins' estimate (Lemma~\ref{rhou}).

\begin{proposition}
\label{propn: boundG}
For any $\varepsilon\in ]0,1[$, $\beta>1$, and $q>4$, there exists a constant  $C=C(\e,\gamma, \mu, q, E_0)$ such that the commutator term $G$ associated to the strong solution $(\rho, u)$ is bounded as follows:
		\begin{equation*}
			\begin{aligned}
				\|G\|_{L^{\infty}}\leq C\left\{\hatrho^{\,\alpha_1}\left[D^{2-\frac{2}{q}} + D^{2-\frac{6}{q}}\right]
				+\hatrho ^{\,\alpha_2} \left[D^{2-\frac{2}{q}} + D^{2-\frac{6}{q}} \right]\left(\frac{X^2}{10+Y^2}\right)^{\frac{1}{q}}+ \hatrho^{\,\alpha_3} D^{2-\frac{6}{q}}\right\}.	
			\end{aligned}
		\end{equation*}
The indices $\alpha_1, \alpha_2, \alpha_3$ are defined as
\begin{eqnarray*}
&& \alpha_1=\left(\frac{\beta\varepsilon}{2}+\frac{|\gamma-\beta|}{4}\right)\frac{4}{q}+1-\frac{1}{q}+\left(\frac{1}{2}-\frac{1}{q}\right)\varsigma,\\
&& \alpha_2=\left(\frac{\beta\e}{2}+ \max\left\{\frac{\gamma-\beta}{2},0\right\}\right)\frac{4}{q}+1+\left(\frac{1}{2}-\frac{1}{q}\right)\varsigma,\\
&& \alpha_3=\max\left\{0,\frac{3\gamma}{4}-\beta\right\}\frac{4}{q}+1-\frac{1}{q}+\left(\frac{1}{2}-\frac{1}{q}\right)\varsigma,
\end{eqnarray*}
where, as before, $\varsigma:=1+\beta\varepsilon+\max\left\{0, \gamma-2\beta,\beta-\gamma-2\right\}$.
	\end{proposition}

\begin{proof}[Proof of Proposition~\ref{propn: boundG}]
For any $q>4$, we may estimate as follows:
		\begin{align}\label{G, prelim}
			\|G\|_{L^{\infty}}&\leq C(q)\|\nabla G\|_{L^{\frac{4q}{q+4}}}^{\frac{4}{q}}\|G\|_{L^q}^{1-\frac{4}{q}} \nonumber \\
			&\leq C(q)\|\nabla u\|_{L^4}^{\frac{4}{q}}\|\rho u\|_{L^q}^{\frac{4}{q}}\|u\|_{{\mathcal{BMO}}}^{1-\frac{4}{q}}\|\rho u\|_{L^q}^{1-\frac{4}{q}}\nonumber \\
			&\leq C(q)\|\nabla u\|_{L^4}^{\frac{4}{q}}\|\rho u\|_{L^q}^{\frac{4}{q}}\|\nabla u\|_{L^2}^{1-\frac{4}{q}}\|\rho u\|_{L^q}^{1-\frac{4}{q}}\nonumber \\
			&= C(q)\|\nabla u\|_{L^4}^{\frac{4}{q}}\|\nabla u\|_{L^2}^{1-\frac{4}{q}}\|\rho u\|_{L^q}.
		\end{align}
The first line holds by the Gagliardo--Nirenberg--Sobolev interpolation inequality, the second line by the commutator estimates in Lemma~\ref{commu} (\emph{cf}. Coifman--Rochberg--Weiss\cite{Coifman} and Coifman--Meyer \cite{Coifman1}), and the third by the continuous embedding $\dot{W}^{1,2}(\tor) \emb \mathcal{BMO}(\tor)$ (see \eqref{bmo2}).

It now remains to bound $\|\rho u\|_{L^q}$. Clearly, one has that 
\begin{align*}
    \|\rho u\|_{L^q} \leq \tilde{\rho}^{1-\frac{1}{q}}(t)\left\|\rho^{\frac{1}{q}} u\right\|_{L^q}.
\end{align*}
We apply Desjardins' estimate in Lemma~\ref{rhou} to control the right-most term. This leads to
\begin{align*}
  \left\|\rho^{\frac{1}{q}} u\right\|_{L^q} \leq   C(q)\|\sqrt{\rho}u\|_{L^2}^{\frac{2}{q}} \|\nabla u\|_{L^2}^{1-\frac{2}{q}}\left\{\log \left(2+\frac{\|\nabla u\|_{L^2}^2 \|\rho\|_{L^{\gamma}}}{\|\sqrt{\rho}u\|_{L^2}^{2}}\right)\right\}^{\frac{1}{2}-\frac{1}{q}}.
\end{align*}
In addition, by virtue of the energy inequality in Lemma~\ref{lemma: energy ineq}, the terms $\|\rho\|_{L^\gamma}$ and $\|\sqrt{\rho}u\|_{L^2}$ are controlled by the initial energy $E_0$. Summarising the above arguments, we arrive at 
\begin{align}\label{G, prelim, 2}
     \|\rho u\|_{L^q}\leq C(q,\gamma, E_0) \tilde{\rho}^{1-\frac{1}{q}}(t) \|\nabla u\|_{L^2}^{1-\frac{2}{q}}\Big\{\log \left(2+\|\nabla u\|_{L^2}^2\right)\Big\}^{\frac{1}{2}-\frac{1}{q}}.
\end{align}

Finally, for the term $\|\na u\|_{L^4}$  in \eqref{G, prelim}, we apply Lemma~\ref{lem: Lp for nabla u} (reproduced below): for each $q'>2$ there exists a constant $C=C(\mu,\gamma, q',\varepsilon, E_0)$ such that:
\begin{equation*}
			\begin{aligned}
				\|\nabla u\|_{L^{q'}}&\leq C \tilde{\rho} ^{\frac{\beta\varepsilon}{2}+\max\left\{0,\frac{\gamma-\beta}{2}\right\}\cdot\frac{2}{q'}+\max\left\{0,\frac{\beta-\gamma}{2}\right\}\cdot(1-\frac{2}{q'})}\big[1+D\big]
    \\&\qquad+ C\tilde{\rho} ^{\frac{\beta\varepsilon}{2}+\frac{1}{2}-\frac{1}{q'}+\max\left\{0,\frac{\gamma-\beta}{2}\right\}}\big[1+D\big]\left(\frac{X^2}{10+Y^2}\right)^{\frac{1}{2}-\frac{1}{q'}} + C \tilde{\rho}^{\max\left\{0,\frac{q'-1}{q'}\gamma-\beta\right\}}.
			\end{aligned}
		\end{equation*}
In particular, when $q'=4$ this yields that
\begin{align*}
    \|\na u\|_{L^{4}} &\leq C \tilde{\rho} ^{\frac{\beta\varepsilon}{2}+\max\left\{0,\frac{\gamma-\beta}{4}\right\}+\max\left\{0,\frac{\beta-\gamma}{4}\right\}}\big[1+D\big]\\
    &\quad + C\trho^{\frac{\beta\e}{2}+\frac{1}{4} +\max\{0,\frac{\gamma-\beta}{2}\}}\big[1+D\big]\left(\frac{X^2}{10+Y^2}\right)^{\frac{1}{4}} + C\trho^{\max\left\{ 0, \frac{3\gamma}{4}-\beta \right\}}.
\end{align*}
This together with the estimates in \eqref{G, prelim} and \eqref{G, prelim, 2} leads to
\begin{align*} 
\|G\|_\linf &\leq C\Bigg\{ \trho^{\frac{\beta\e}{2}+\frac{|\beta-\gamma|}{4}}[1+D] + \trho^{\frac{\beta\e}{2}+\frac{1}{4} +\max\{0,\frac{\gamma-\beta}{2}\}}\big[1+D\big]\left(\frac{X^2}{10+Y^2}\right)^{\frac{1}{4}} + \trho^{\max\left\{ 0, \frac{3\gamma}{4}-\beta\right\}}
    \Bigg\}^{\frac{4}{q}}  \\
    &\qquad \times D^{1-\frac{4}{q}} \trho^{1-\frac{1}{q}} D^{1-\frac{2}{q}} \Big\{\log \left(2+\|\nabla u\|_{L^2}^2\right)\Big\}^{\frac{1}{2}-\frac{1}{q}},
\end{align*}
with some constant $C=C(\mu,\gamma, q,\varepsilon, E_0)$. Hence, by virtue of Proposition \ref{propn: logY}, we further have that
 \begin{align*}   
\|G\|_\linf & \leq C(\mu,\gamma, q,\varepsilon, E_0)\Big\{ \hatrho^{(\frac{\beta\e}{2}+\frac{|\beta-\gamma|}{4})\cdot\frac{4}{q}+1-\frac{1}{q}+\varsigma\cdot(\frac{1}{2}-\frac{1}{q})}[1+D]^{\frac{4}{q}}D^{2-\frac{6}{q}} \\
    & \qquad+\hatrho^{(\frac{\beta\e}{2}+\frac{1}{4} + \max\left\{\frac{\gamma-\beta}{2},0\right\})\cdot\frac{4}{q}+1-\frac{1}{q}+\varsigma\cdot(\frac{1}{2}-\frac{1}{q})} [1+D]^{\frac{4}{q}}\left(\frac{X^2}{10+Y^2}\right)^{\frac{1}{q}}D^{2-\frac{6}{q}}\\
    &\qquad +\hatrho^{\max\left\{ 0, \frac{3\gamma}{4}-\beta\right\}\cdot \frac{4}{q}+1-\frac{1}{q}+\varsigma\cdot(\frac{1}{2}-\frac{1}{q})}D^{2-\frac{6}{q}}
    \Big\}\\
   & \leq C(\mu,\gamma, q,\varepsilon, E_0)\Big\{ \hatrho^{(\frac{\beta\e}{2}+\frac{|\beta-\gamma|}{4})\cdot\frac{4}{q}+1-\frac{1}{q}+\varsigma\cdot(\frac{1}{2}-\frac{1}{q})}\left[D^{2-\frac{2}{q}} + D^{2-\frac{6}{q}}\right] \\
    & \qquad+\hatrho^{(\frac{\beta\e}{2}+\frac{1}{4} + \max\left\{\frac{\gamma-\beta}{2},0\right\})\cdot\frac{4}{q}+1-\frac{1}{q}+\varsigma\cdot(\frac{1}{2}-\frac{1}{q})} \left[D^{2-\frac{2}{q}} + D^{2-\frac{6}{q}} \right]\left(\frac{X^2}{10+Y^2}\right)^{\frac{1}{q}}\\
    &\qquad +\hatrho^{\max\left\{ 0, \frac{3\gamma}{4}-\beta\right\}\cdot \frac{4}{q}+1-\frac{1}{q}+\varsigma\cdot(\frac{1}{2}-\frac{1}{q})}D^{2-\frac{6}{q}}
    \Big\}.
\end{align*}
The assertion immediately follows from the definition of $\alpha_1$, $\alpha_2$, and $\alpha_3$.   
 \end{proof}

With the above preparations at hand (in particular, the $L^q$-bound for the gradient of $u$ in Lemma~\ref{lem: Lp for nabla u}, the estimate for $\log \|\na u\|_{L^2}$ in Proposition~\ref{propn: logY}, and the $\linf$-bound for the commutator term $G$ in Proposition~\ref{propn: boundG}), we are at the stage of proving our main Theorem~\ref{thm: main}. 
This is the main content of the remaining two sections:
\begin{itemize}
    \item 
In \S\ref{sec: rho bound} we prove the non-formation of vacuum or concentration of density at any finite time when $\beta>\frac{3}{2}$ and $1<\gamma<4\beta-3$, and then deduce the large-time behaviour of $\rho$ and $\na u$.
\item 
In \S\ref{sec: global existence} we prove the existence of weak solutions up to any finite time when $\beta>1$ and $\gamma>1$.
\end{itemize}

	\section{Upper and lower bounds for density $\rho$ and the large-time behaviour}\label{sec: rho bound}

We now establish the non-formation of vacuum or concentration, and deduce from it the large-time behaviour as stated in Theorem~\ref{thm: main}. Throughout this section, assume
\begin{equation}\label{3/2 assumption}
    \beta>\frac{3}{2}\qquad \text{and}\qquad 1<\gamma<4\beta-3.
\end{equation}
Also recall from Equation~\eqref{F} the quantity $F$ first introduced in Desjardins \cite{des2}: $$F:=2\mu \log\rho +\frac{\rho^{\beta}}{\beta}-(-\Delta)^{-1} \mathrm{div}(\rho u).$$

The strategy of the proof  follows Perepelitsa \cite{perep} in the large.

\subsection{Upper bound on density}

We first prove the non-formation of concentration of density.

Note that for $\rho>0$ one has $$P(\rho)\geq \gamma \log\rho +1.$$ We then have the differential inequality for $F^+:=\max\{F,0\}$:
      \begin{equation}
		\label{dF}
			\frac{\mathrm{D}F^{+}}{\mathrm{D}t}+\frac{\gamma}{2\mu}F^{+} \leq \frac{\gamma}{2\mu}\left\|(-\Delta)^{-1}\mathrm{div}(\rho u)\right\|_{L^{\infty}}+\frac{\gamma}{2\mu\beta} \rho^{\beta}+ \|G\|_{L^{\infty}}+\left|\bar{P}-\bar{B}\right|.
		\end{equation}
Here ${\rm D}/{\rm D}t$ is the material derivative.

The $\linf$-norm of the term $(-\Delta)^{-1}\mathrm{div}(\rho u)$ on the right-hand side can be controlled as follows (here the parameter $q \in ]2,\infty[$ is arbitrary):
\begin{align}\label{xxx, Nov24}
&\left\|(-\Delta)^{-1}\mathrm{div}(\rho u)\right\|_{L^{\infty}}\nonumber\\
&\qquad \leq C\left\|(-\Delta)^{-1}\nabla\mathrm{div}(\rho u)\right\|_{L^2}\Big\{\log(e+\|\rho u\|_{L^q})\Big\}+ C\left\|(-\Delta)^{-1}\mathrm{div}(\rho u)\right\|_{L^2}\nonumber\\
			&\qquad \leq C\|\rho u\|_{L^2}\Big\{\log(e+\|\rho u\|_{L^q})\Big\}^{\frac{1}{2}}+C\|\rho u\|_{L^2}\nonumber\\
			&\qquad \leq C\tilde{\rho}^{\frac{1}{2}}(t)\|\sqrt{\rho}u\|_{L^2} \left\{\log\left(e+\tilde{\rho}^{1-\frac{1}{q}}(t)\|\nabla u\|_{L^2}^{1-\frac{2}{q}}\left(\log (2+\|\nabla u\|_{L^2}^2)\right)^{\frac{1}{2}-\frac{1}{q}}\right)\right\}^{\frac{1}{2}}\nonumber\\
               &\qquad \leq C\tilde{\rho}^{\frac{1}{2}}(t)\|\sqrt{\rho}u\|_{L^2} \left\{\log\left(e+\tilde{\rho}^{1-\frac{1}{q}}(t)\|\nabla u\|_{L^2}^{1-\frac{2}{q}}\hatrho^{\,(\frac{1}{2}-\frac{1}{q})\varsigma}\right)\right\}^{\frac{1}{2}}\nonumber\\
                &\qquad \leq C(E_0)\tilde{\rho}^{\frac{1}{2}}(t) \left\{\log\left(e+\|\nabla u\|_{L^2}^{2}\right)\right\}^{\frac{1}{2}}+C(E_0)\hatrho^{\,1-\frac{1}{2q}+(\frac{1}{4}-\frac{1}{2q})\varsigma}(t)\nonumber\\
			&\qquad \leq C(E_0)\hatrho^{\,\alpha_4}(t).
\end{align}
In the above, the first line follows from the Brezis--Wainger inequality (Lemma~\ref{lemma: brezis-wainger}); in the second line we use the fact that the average of $\rho u$ is zero and the continuity of the Riesz transform; and the remaining lines hold by Lemma~\ref{rhou} and Proposition~\ref{propn: logY}. Here and hereafter, we set 	$$\alpha_4:=\max\left\{\frac{1}{2}+\frac{1}{2}\varsigma, 1-\frac{1}{2q}+\left(\frac{1}{4}-\frac{1}{2q}\right)\varsigma\right\}.$$

Moreover, thanks to the energy inequality in Lemma~\ref{lemma: energy ineq} and the definition of $B$ in \eqref{defB}, one has the following estimate for the right-most term in \eqref{dF}:
         \begin{align}
         \label{pb}
            \left|\bar{P}-\bar{B}\right|&\leq \int_{\mathbb{T}^2} \rho^{\gamma}\,\mathrm{d}x + \left\{\int_{\mathbb{T}^2} (\lambda+2\mu )\,\mathrm{d}x\right\}^{\frac{1}{2}}\left\{\int_{\mathbb{T}^2} (\lambda+2\mu)|\mathrm{div}u|^2 \,\mathrm{d}x\right\}^{\frac{1}{2}} \nonumber\\
            &\leq C(E_0) + \left\{\tilde{\rho}^{\max\left\{0,\frac{\beta-\gamma}{2}\right\}}(t) \left(\int_{\mathbb{T}^2} \rho^{\gamma}\,\mathrm{d}x \right)^{\frac{1}{2}} +\sqrt{2\mu}\right\} D(t)\nonumber\\
            &\leq C(E_0)+C(E_0,\mu)\left(\tilde{\rho}^{\max\left\{0,\frac{\beta-\gamma}{2}\right\}}(t)+1\right)D(t).
         \end{align}

Now, standard maximum principle arguments applied to the differential inequality~\eqref{dF} yields that
		\begin{equation*}
			\begin{aligned}
	\left\|F^{+}(t)\right\|_{L^{\infty}}\leq \,& \,e^{-\frac{\gamma}{2\mu}}\left\|F^{+}(0)\right\|_{L^{\infty}}+\frac{\gamma}{2\mu}\int_{0}^{t}e^{-\frac{\gamma}{2\mu}(t-\tau)}\left\|(-\Delta)^{-1}\mathrm{div}(\rho u)\right\|_{L^{\infty}}(\tau)\,\mathrm{d}\tau\\ &+\frac{\gamma}{2\mu\beta}\int_{0}^{t}e^{-\frac{\gamma}{2\mu}(t-\tau)}\rho^{\beta}(\tau)\,\mathrm{d}\tau
+\int_{0}^{t} e^{-\frac{\gamma}{2\mu}(t-\tau)}\|G\|_{L^{\infty}}(\tau)\,\mathrm{d}\tau\\
&+\int_{0}^{t}e^{-\frac{\gamma}{2\mu}(t-\tau)}\left|\bar{P}-\bar{B}\right|(\tau)\,\mathrm{d}\tau.
			\end{aligned}
		\end{equation*}
From Equation~\eqref{pb},  Proposition~\ref{propn: boundG}, and the definition of $F$ (Equation~\eqref{F}),  we then infer that
\begin{align}\label{hatrho}
\hatrho^{\,\beta}(t) & \leq e^{-\frac{\gamma}{2\mu}}\left\|F^{+}(0)\right\|_{L^{\infty}}+C\int_{0}^{t}e^{-\frac{\gamma}{2\mu}(t-\tau)}\hatrho^{\,\alpha_4}(\tau)\,\mathrm{d}\tau +C\int_{0}^{t}e^{-\frac{\gamma}{2\mu}(t-\tau)}\hatrho^{\,\beta}(\tau)\,\mathrm{d}\tau \nonumber\\
&\quad + C\int_{0}^{t} e^{-\frac{\gamma}{2\mu}(t-\tau)}\Bigg\{\hatrho^{\,\alpha_1}\left[D^{2-\frac{2}{q}} + D^{2-\frac{6}{q}}\right]+\hatrho ^{\alpha_2} \left[D^{2-\frac{2}{q}} + D^{2-\frac{6}{q}} \right]\left(\frac{X^2}{10+Y^2}\right)^{\frac{1}{q}} + \hatrho^{\,\alpha_3} D^{2-\frac{6}{q}}\Bigg\}\,\mathrm{d}\tau\nonumber\\
&\quad
+\int_{0}^{t}e^{-\frac{\gamma}{2\mu}(t-\tau)}\hatrho^{\,\max\left\{0,\frac{\beta-\gamma}{2}\right\}}(\tau)D(\tau)\,\mathrm{d}\tau+C\int_{0}^{t}e^{-\frac{\gamma}{2\mu}(t-\tau)}D(\tau)\,\mathrm{d}\tau,
\end{align}
with constants $C$ depending only on $\e$, $q$, $E_0$, $\gamma$, $\mu$, and $\beta$.

The second line in \eqref{hatrho} above can be controlled via the energy inequality (Lemma~\ref{lemma: energy ineq}), H\"{o}lder's inequality, and Proposition \ref{propn: logY}. Indeed, we have that
\begin{align*}
&\int_{0}^{t} e^{-\frac{\gamma}{2\mu}(t-\tau)}\hatrho^{\,\alpha_1}(\tau)D^{2-\frac{2}{q}}\,\mathrm{d}\tau\\ 
&\qquad \leq \hatrho^{\,\alpha_1}(t)\left\{\int_{0}^{t} e^{-\frac{\gamma}{2\mu}(t-\tau)}D^{2}\,\mathrm{d}\tau\right\}^{1-\frac{1}{q}}\left\{\int_{0}^{t} e^{-\frac{\gamma}{2\mu}(t-\tau)}\,\mathrm{d}\tau\right\}^{\frac{1}{q}}\\
&\qquad\leq C(E_0,\gamma,\mu,q)\hatrho^{\,\alpha_1}(t)
        \end{align*}      
        and that
\begin{align*}
&\int_{0}^{t} e^{-\frac{\gamma}{2\mu}(t-\tau)} \hatrho ^{\alpha_2}(\tau) D^{2-\frac{2}{q}}\left(\frac{X^2}{10+Y^2}\right)^{\frac{1}{q}}\,\mathrm{d}\tau \\
&\qquad \leq \hatrho^{\,\alpha_2}(t)\left\{\int_{0}^{t} e^{-\frac{\gamma}{2\mu}(t-\tau)}D^{2}\,\mathrm{d}\tau\right\}^{1-\frac{1}{q}}\left\{\int_{0}^{t} e^{-\frac{\gamma}{2\mu}(t-\tau)}\frac{X^2}{10+Y^2}\,\mathrm{d}\tau\right\}^{\frac{1}{q}}\\
& \qquad\leq C(E_0,\gamma,\mu,q) \hatrho^{\,\alpha_2+\frac{1}{q}\varsigma}(t).
        \end{align*}
Thus, one deduces from Equation~\eqref{hatrho} that
\begin{equation}
           \label{rho1}
			\begin{aligned}
				\hatrho^{\,\beta}(t)
               \leq & \,C+C\int_{0}^{t}e^{-\frac{\gamma}{2\mu}(t-\tau)}\hatrho^{\,\beta}(\tau)\,\mathrm{d}\tau+C\hatrho^{\,\alpha_1}(t)+C\hatrho^{\,\alpha_2+\frac{1}{q}\varsigma}(t)\\&+C\hatrho^{\,\alpha_3}(t)+C\hatrho^{\,\max\left\{0,\frac{\beta-\gamma}{2}\right\}}(t)+C\hatrho^{\,\alpha_4}(t),
		   \end{aligned}
		\end{equation}
with the constants $C=C(\e,q, E_0, \gamma, \mu, \beta)$.

By a standard Gr\"{o}nwall argument, the global-in-time bound for $\hatrho$ can be deduced from \eqref{rho1} only when the powers of $\hatrho$ on the right-hand side are strictly less than $\beta$ (except for the term $\hatrho^{\,\beta}$ under the integral sign). That is, we require  
\begin{equation}\label{beta, Nov24}
    \max\left\{\alpha_1,\alpha_2+\frac{1}{q}\varsigma,\alpha_3,\alpha_4 \right\}<\beta.
\end{equation}
In view of the definitions of $\alpha_i$, $i \in\{1,2,3,4\}$, this is possible (by choosing $q$ sufficiently large) if the disjunction of the following four conditions holds:
   \begin{itemize}
       \item $\gamma<\beta<\gamma+2$ and $\frac{3}{2}<\beta$;
       \item $\gamma+2<\beta$ and $1<\beta$;
       \item $\beta<\gamma<2\beta$ and $\frac{3}{2}<\beta$;
       \item $2\beta<\gamma$ and $\gamma<4\beta-3$
   \end{itemize}
which, in turn, is valid when the condition~\eqref{3/2 assumption} in the statement of Theorem~\ref{thm: main} holds.

Therefore, assuming the condition~\eqref{3/2 assumption}, we deduce from \eqref{rho1} that
        \begin{equation}
            \label{rho2}
				\hatrho^{\beta}(t)
				\leq C+C\int_{0}^{t}e^{-\frac{\gamma}{2\mu}(t-\tau)}\hatrho^{\,\beta}(\tau)\,\mathrm{d}\tau.
        \end{equation}
Hence, by Gr\"{o}nwall's inequality we have that $\hatrho^{\,\beta}(t)\leq C=C(\e,q, E_0, \gamma, \mu, \beta,M)$, where ${M} <\infty$ is the essential supremum of the initial density $\rho_0$. That is, there exists a time-independent positive constant $$\bar{M}=\bar{M}\Big(\e,q, E_0, \gamma, \mu, \beta, M\Big)$$ such that
	\begin{equation}
        \label{uprho}
	    \rho(t,x)\leq \bar{M}\qquad\text{for a.e. } (t,x)  \in ]0,+\infty[\times\mathbb{T}^2.
	\end{equation}
This proves the non-formation of concentration of density.

\subsection{Time-dependent lower bound on density}
\label{lower bound}
We now turn to the non-formation of vacuum, \emph{i.e.}, the lower bound on density away from zero, up to any finite time.

For this purpose, one considers the quantity 
\begin{equation}\label{theta, def}
\theta(\rho):= 2\mu \log\rho +\frac{\rho^{\beta}}{\beta},  
\end{equation}
which is closely related to $F$ defined in \eqref{F}.

The conservation of momentum expressed in terms of $F$, \emph{i.e.}, Equation~\eqref{eq1}, leads to the following differential inequality (\emph{cf.} Perepelitsa \cite{perep}):
	\begin{equation}
		\label{tra}
		\frac{\mathrm{D}}{\mathrm{D}t}\theta \geq \frac{\mathrm{D}}{\mathrm{D}t}\left[(-\Delta)^{-1}\mathrm{div}(\rho u)\right] +\bar{P}-\left|G+P+\bar{B}\right|.
	\end{equation}
We \emph{claim} that for any $T>0$,  \begin{align}
     \label{claim, June25}
 \int_{0}^{T} \left\{\bar{P}-|G+P+\bar{B}|\right\}(\tau, x)\,\mathrm{d}\tau \geq - C\left(\bar{M},q, E_0\right)-T.  
\end{align}	
To see this, by estimates for $G$ in Proposition~\ref{propn: boundG} and those for $\rho$ in Equation~\eqref{uprho}, we obtain for \emph{a.e.} $t\in [0,T]$ that
\begin{align*}
\sup_{x\in\mathbb{T}^2}|G(t,x)| \leq C(\bar{M})\left\{\left[D^{2-\frac{2}{q}} + D^{2-\frac{6}{q}}\right] + \left[D^{2-\frac{2}{q}} + D^{2-\frac{6}{q}} \right]\left(\frac{X^2}{10+Y^2}\right)^{\frac{1}{q}}+  D^{2-\frac{6}{q}}\right\}(t).    
\end{align*}
On the other hand, by virtue of $P=\rho^\gamma$ and the definition of $B$ in \eqref{defB}, we deduce from the upper bound~\eqref{uprho} for $\rho$ and the Cauchy--Schwarz inequality that
      $$\begin{aligned}
        \sup_{x\in \mathbb{T}^2}|P(t,x)|+\left|\bar{B}(t)\right|&\leq C+\left\{\int_{\mathbb{T}^2} (\lambda+2\mu )\,\mathrm{d}x\right\}^{\frac{1}{2}}\left\{\int_{\mathbb{T}^2} (\lambda+2\mu)|\mathrm{div}u|^2 \,\mathrm{d}x\right\}^{\frac{1}{2}}\\
        &\leq C\big(1+D(t)\big),
      \end{aligned}$$
where $C=C(\bar{M},\gamma,\beta,\mu)$.   Hence, by Young's inequality, 
    \begin{equation}
    \label{GPB}    \sup_{x\in\mathbb{T}^2}\left|G(t,x)+P(t,x)+\bar{B}(t)\right|\leq  C\left(\bar{M},q\right)\left\{1+D^2(t)+\frac{X^2(t)}{10+Y^2(t)}\right\}.
    \end{equation}
For \emph{a.e.} $t \in [0,T]$, it thus holds that 
\begin{align*}
 \int_{0}^{T} \left\{\bar{P}-\left|G+P+\bar{B}\right|\right\}(t, x)\,\mathrm{d}t &\geq -\int_{0}^{T} \left\{\sup_{x\in\mathbb{T}^2}\left|G+P+\bar{B}\right|\right\}(t,x)\,\mathrm{d}t\nonumber\\     
          &\geq -\int_{0}^{T} C\left(\bar{M},q\right)\left\{1+D^2(t) + \frac{X^2(t)}{10+Y^2(t)}\right\} \,\mathrm{d}t\nonumber\\
         &\geq - C\left(\bar{M},q, E_0\right)-T.
\end{align*}
The penultimate inequality follows from \eqref{GPB}, and the final line holds by the energy inequality in Lemma~\ref{lemma: energy ineq} and the estimate~\eqref{log Y bound} in the proof of Proposition~\ref{propn: logY}. This proves the \emph{claim}~\eqref{claim, June25}.

Integrating the differential inequality~\eqref{tra} over time, we deduce for almost every $t>0$ that 
\begin{align}\label{integrated diff ineq for theta}
\theta\big(\rho(t,x)\big)& \geq\theta\big(\rho_0(x)\big) -\left\|(-\Delta)^{-1}\mathrm{div}(\rho u)(t,\bullet)\right\|_{L^{\infty}}\nonumber \\
    &\qquad +\int_{0}^{t} \left\{\bar{P}-\left|G+P+\bar{B}\right|\right\}(\tau, x)\,\mathrm{d}\tau.
\end{align}
The third term on the right-hand side is controlled via  \eqref{claim, June25}. For the second term, via \eqref{xxx, Nov24} it is bounded from below by $-C(E_0)\hatrho^{\,\alpha_4}$ with $\alpha_4:=\max\left\{\frac{1}{2}+\frac{1}{2}\varsigma, 1-\frac{1}{2q}+\left(\frac{1}{4}-\frac{1}{2q}\right)\varsigma\right\}$, and by \eqref{beta, Nov24} we have $\alpha_4<\beta$ for suitably chosen $q$ under the assumption~\eqref{3/2 assumption}. Thus, in view of \eqref{uprho}, 
\begin{equation}\label{y, June25}
    \theta\big(\rho(t,x)\big) \geq\theta\big(\rho_0(x)\big) - C\Big(\e,q, E_0, \gamma, \mu, \beta,M\Big)-T.
\end{equation}

Thus, noting that $\theta(\rho)$ behaves like $\log\rho$ when $\rho$ becomes close to zero and that $\rho_0>m>0$ \emph{a.e.}, we conclude that
	\begin{equation}
        \label{lowrho}
		\rho(t,x)\geq \bar{m}\Big(\e,q, E_0, \gamma, \mu, \beta, M, m\Big)e^{-T} > 0 \qquad \text{for a.e. } (t,x)\in [0,T]\times \mathbb{T}^2.
	\end{equation}
This concludes the proof for the lower bound of $\rho$.

\begin{remark}\label{remark: M, m dependency}
In the derivation of the upper and lower bounds $\left(\bar{M}, \bar{m}\right)$ for density, \emph{i.e.}, Equations~\eqref{uprho} and \eqref{lowrho}, the parameter $\e$ enters via the variable $\varsigma$ (Proposition~\ref{propn: logY}). In addition, we have first chosen $\e$ suitably small and then $q$ suitably large to ensure the validity of \eqref{beta, Nov24}. Thus, we may view $\e$ and $q$ as fixed once and for all, and hence take $\bar{M}=\bar{M}(E_0, \gamma, \mu,\beta, M)$ and  $\bar{m}=\bar{m}(E_0, \gamma, \mu,\beta, M, m)$ from now on, where $M,m$ are the essential supremum and infimum of the initial density $\rho_0$, respectively.

\end{remark}

\subsection{Large-time behaviour}\label{subsec: large-time}
With the upper and lower bounds for $\rho$ at hand, we now establish the large-time behaviour of global weak solutions under the assumption~\eqref{3/2 assumption}, namely $\beta>\frac{3}{2}$ and $1<\gamma<4\beta-3$. Our arguments are essentially taken from Huang--Li~\cite{H1}.

First, observe that  
\begin{equation}\label{X2, Nov24}\int_{0}^{\infty}X^2(t)\,\mathrm{d}t\leq C
\end{equation}
in light of Proposition ~\ref{propn: logY} and the upper bound~\eqref{uprho} for $\rho$. Thanks to Remark~\ref{remark: M, m dependency}, we may take $C=C(\mu,\beta,\gamma,E_0)$ here.  Notice too that \begin{equation}\label{D2, Nov24}\int_{0}^{\infty}D^2(t)\,\mathrm{d}t\leq C'
\end{equation}
by the definition of $D$ in \eqref{D, def} and the energy inequality in Lemma~\ref{lemma: energy ineq}; here $C'=C'(\gamma,E_0)$.

We next \emph{claim} that
	\begin{equation}
		\label{intPP}
\int_{0}^{\infty}\left\|\left(P-\bar{P}\right)(t,\bullet)\right\|_{L^{2}}^2\,\mathrm{d}t \leq  C\left(\gamma,\mu,\beta,E_0\right)
	\end{equation}
where $P=\rho^\gamma$. Hence,
\begin{align}\label{limP}
    \lim_{t\to\infty}\left\|\left(P-\bar{P}\right)(t,\bullet)\right\|_{L^{2}} = 0.
\end{align}  
To see this, we use the definition of $B$, $D$ (Equations~\eqref{defB} and \eqref{D, def}) and the upper bound for $\rho$ in~\eqref{uprho} to deduce that
\begin{align*}
\left\|\left(P-\bar{P}\right)(t,\bullet)\right\|_{L^2}^2  &\leq 2 \left\|B(t,\bullet)\right\|_{L^2}^2 + 2\left\|\left(\lambda(\rho)+2\mu\right){\rm div}u(t,\bullet)\right\|^2_{L^2}\\
&\leq 2 \left\|B(t,\bullet)\right\|_{L^2}^2 + C\left(\mu,\bar{M}\right)D^2(t).
\end{align*}
By \eqref{uprho} again and the estimate for $\|B\|_{H^1}$ in \eqref{BH1}, we have that 
\begin{align*}
 \left\|B(t,\bullet)\right\|_{L^2}^2 \leq C\left(\beta,\gamma,\mu,E_0,\bar{M}\right)\big(X^2(t)+D^2(t)\big).
\end{align*}
Hence,
\begin{align*}
\left\|\left(P-\bar{P}\right)(t,\bullet)\right\|_{L^2}^2 &\leq  C\left(\beta,\gamma,\mu,E_0,\bar{M}\right)\big(X^2(t)+D^2(t)\big).
\end{align*}
The \emph{claim}~\eqref{intPP} follows from integrating the above inequality over $t \in [0,\infty[$ and invoking \eqref{X2, Nov24} and \eqref{D2, Nov24}. See Remark~\ref{remark: M, m dependency} for the dependency of parameters.

On the other hand, as in \cite[p.152]{H1}, we deduce from Equation~\eqref{x, June25} for $P$,  Cauchy--Schwarz, the \emph{claim}~\eqref{intPP}, as well as the bounds in \eqref{X2, Nov24}, \eqref{D2, Nov24}, and $\|{\rm div}u\|_{L^2}\leq C(\mu)D$ that
\begin{align*}
\int_0^\infty \left|\frac{\dd \bar{P}}{\dd t}\right|\,\dd t &\leq C(\gamma)\int_0^\infty\left|\int_\tor \left(P-\bar{P}\right){\rm div} u\,\dd x\right|\,\dd t\\
&\leq C(\gamma,\mu)\int_0^\infty\left\{ \left\|\left(P-\bar{P}\right)(t,\bullet)\right\|_{L^2}^2 + D(t)^2 \right\}\,\dd t\\
&\leq C(\mu,\beta,\gamma,E_0).
\end{align*}
This leads to $\bar{P} \in \dot{W}^{1,1}(\R_+)$ and hence $$\lim_{t \to \infty}\bar{P}(t) = a^\gamma$$ for some positive number $a$, since   at any time $t$ one has by Jensen's inequality that $$\bar{P}(t) \geq \left(\int_\tor\rho(t,x)\,\dd x\right)^\gamma \equiv \overline{\rho_0}^\gamma>0.$$ 
On the other hand, as $\lim_{t\to\infty}(P-\bar{P})(t,\bullet) \to 0$ in $L^2(\tor)$, along a sequence $\{t_j\}\nearrow\infty$ we have $P(t_j,\bullet) \to a^\gamma$ \emph{a.e.} and hence $\rho(t_j,\bullet) \to a$ on $\tor$. Integrating over $\tor$ and applying the dominated convergence theorem thanks to the upper bound~\eqref{uprho} for $\rho$, we find that $a=\overline{\rho_0}$.

In view of Equation~\eqref{limP},  the upper bound~\eqref{uprho} for $\rho$, and the previous paragraph, we conclude that 
\begin{equation}\label{Lp conv for rho}   \lim_{t\rightarrow\infty}\left\|\rho(t,\bullet)-\overline{\rho_0}\right\|_{L^p}=0,\qquad  \text{for any } 1\leq p<\infty.
\end{equation}
Recall from the equation ensuing~\eqref{mean} the normalisation $\overline{\rho_0}=1$.

Finally, in view of the energy inequality in Lemma~\ref{lemma: energy ineq}, we have $\int_0^{\infty}D^2(t)\,\mathrm{d}t\leq C(E_0)$ and hence  $\lim\limits_{t\rightarrow\infty} D^2(t)=0$. Thus, in view of the definition of $D$ in \eqref{D, def}, one concludes that
	\begin{align}
		\label{limu}	\lim_{t\rightarrow\infty}\|\nabla u(t,\bullet)\|_{L^2} &\leq \lim_{t\rightarrow\infty}(\|\mathrm{div}\,u(t,\bullet)\|_{L^2}+\|\mathrm{rot}\,u(t,\bullet)\|_{L^2}) \nonumber\\
        &\leq C\lim_{t\rightarrow\infty}D^2(t)=0.
	\end{align}

 \section{Global existence of weak solutions for $\beta>1$}\label{sec: global existence}

In this final section, we prove the existence of weak solutions of the Navier--Stokes Equation~\eqref{equ} for compressible barotropic fluids up to any finite time $T \in ]0,\infty[$, assuming only 
\begin{align*}
    \beta>1\quad\text{ and }\quad \gamma >1,
\end{align*}
the technical condition~($\clubsuit$) in Theorem~\ref{thm: main}, and that the initial density is away from both vacuum and concentration. As in \cite{vaigant, H1, perep}, the key is to establish the (in this case, time-dependent) upper and lower bounds for the density $\rho$ up to time $T$. The main ingredients of the proof are as follows:
\begin{enumerate}
    \item 
The time-dependent density upper bound in $L^p$-norm for any $p \in [1,\infty[$; \emph{cf.} Va\u{\i}gant--Kazhikhov, \cite[p.1119, Equation~(36)]{vaigant}: 
 \begin{lemma}
    \label{lem:rhoLp}
    Let $\beta>1$ and $1\leq p<\infty$ be arbitrary. There is a constant $C$ depending only on $T$, $\mu$, $\beta$, $\gamma$, $\|\rho_0\|_{L^{\infty}}$, and $\|u_0\|_{L^{\infty}}$, such that
    $$\sup_{t\in [0,T]}\|\rho(t,\bullet)\|_{L^p}\leq C p^{\frac{2}{\beta-1}}.$$
 \end{lemma}

 \item 
The extra assumption~($\clubsuit$) on the higher (namely, $L^{q_1}$ for $q_1>2$) integrability of $\rho u$.

 \item 
The triviality of the limiting measure $\Psi:=\bra \rho^2\ket-\rho^2$; \emph{cf}. Va\u{\i}gant--Kazhikhov \cite[pp.1136--1139]{vaigant} and P.-L. Lions \cite{lions}. Here and hereafter, we  introduce:
\begin{notation}
For a continuous function $f: [0,M]\to\R$ for some finite number $M$, denote by $\bra f(\rho) \ket$ the weak-$\star$ limit of $f\rhod$ as $\delta \to 0^+$, where 
\begin{align*}
\rho^\delta \weak \rho \quad\text{weakly-$\star$ in 
 }  \linf\big(0,T;\linf(\tor)\big). 
\end{align*}

\end{notation}
\end{enumerate}

\subsection{Time-dependent pointwise bound on density}

For the above purpose, we first prove that under the general assumption $\beta>1$ and $\gamma>1$, the density $\rho$ stays away from vacuum and concentration up to any finite time, providing that the initial density $\rho_0$ does so. At the moment, we are unable to obtain the large-time behaviour of $\rho$ and $u$ in this case.

The main result of this subsection is the following:
 \begin{proposition}\label{propn: density bound, beta>1}
Fix $\beta>1$ and $\gamma>1$. Assume the condition~($\clubsuit$) in Theorem~\ref{thm: main}:
\begin{align*}
&\text{There exist an index $q_1>2$, smooth approximations $\{(\rho_n, u_n)\}$ of the solutions, and}\\
&\text{a constant $C_\star$ uniform in $n$ such that } \sup_{n \in \mathbb{N}}\,\sup_{t \geq 0}\, \left\|\rho_n(t,\bullet) u_n(t,\bullet)\right\|_{L^{q_1}(\tor)} \leq C_\star.  \tag{$\clubsuit$}
\end{align*}
Then, for any $T\in ]0,\infty[$, there exist  positive constants $\bar{M}(T)$ and $\bar{m}(T)$ depending only on $T$, $\mu$, $\gamma$, $\beta$, $E_0$, $M$, and $m$, such that
     \begin{equation*}
         0<\bar{m}(T)<\rho(t,x)<\bar{M}(T)<\infty\qquad\text{ for almost all } (t,x)\in [0,T]\times \mathbb{T}^2,
     \end{equation*}
provided that $(\rho, u)$ is a strong solution for the barotropic Navier--Stokes Equation~\eqref{equ}. As before, $M<\infty$ and $m>0$ are the essential supremum and infimum of the initial density $\rho_0$.
 \end{proposition}

 \begin{proof}[Proof of Proposition~\ref{propn: density bound, beta>1}]

We divide our arguments into five steps below.

\smallskip
\noindent
{\bf Step~1.} Recall the estimates for $\|\na u\|_{L^q}$ and $\log Y$ established respectively in Lemma~\ref{lem: Lp for nabla u} and Proposition~\ref{propn: logY}, reproduced below for convenience of the reader:
	\begin{equation*}
			\begin{aligned}
				\|\nabla u\|_{L^q}&\leq C' \tilde{\rho} ^{\frac{\beta\varepsilon}{2}+\max\left\{0,\frac{\gamma-\beta}{2}\right\}\cdot\frac{2}{q}+\max\left\{0,\frac{\beta-\gamma}{2}\right\}\cdot(1-\frac{2}{q})}(t)\big[1+D(t)\big]\\
	  &\qquad+ C'\tilde{\rho} ^{\frac{\beta\varepsilon}{2}+\frac{1}{2}-\frac{1}{q}+\max\left\{0,\frac{\gamma-\beta}{2}\right\}}(t)\big[1+D(t)\big]\left(\frac{X^2(t)}{10+Y^2(t)}\right)^{\frac{1}{2}-\frac{1}{q}}\\
      &\qquad + C' \tilde{\rho}^{\max\left\{0,\frac{q-1}{q}\gamma-\beta\right\}}(t)		
			\end{aligned}
		\end{equation*}
and 
\begin{equation*}
    \log\Big(10 + \|\na u\|^2_{L^2} \Big)\leq C'\hatrho^{\,1+\beta\varepsilon+\max\left\{0, \gamma-2\beta,\beta-\gamma-2\right\}}.
\end{equation*}
Here $q>2$ and $0<\varepsilon<1$ are arbitrary, and the constants $C'=C'(\varepsilon, \mu, \beta, \gamma, M, E_0)$. 

We \emph{claim} that, in the above estimates, all the powers of $\tilde{\rho}$ and $\hatrho$ of the form $\max\{0, \cdots\}$ can be dropped, at the cost of allowing all the constants depend additionally on $T$. That is, we have the estimates as follows:
\begin{equation}
    \label{Lp2}
		\begin{aligned}
				\|\nabla u\|_{L^q}&\leq C \tilde{\rho} ^{\frac{\beta\varepsilon}{2}}(t)\big[1+D(t)\big]+ C\tilde{\rho} ^{\frac{\beta\varepsilon}{2}+\frac{1}{2}-\frac{1}{q}}(t)\big[1+D(t)\big]\left(\frac{X^2(t)}{10+Y^2(t)}\right)^{\frac{1}{2}-\frac{1}{q}}+C
			\end{aligned}
  \end{equation}
  and 
  \begin{equation}
   \label{logY2}   \log\big(10+Y^2(t)\big)+\int_{0}^{t} \frac{X^2(\tau)}{10+Y^2(\tau)}\,\mathrm{d}\tau\leq C \hatrho^{\,1+\beta\varepsilon}(t)\qquad \text{for a.e. } t \in [0,T],
\end{equation}
where $q>2$ and $0<\varepsilon<1$ are arbitrary, and $C=C(T,\varepsilon, \mu, \beta, \gamma, M, E_0)$.

Indeed, an examination of the proofs for Lemma~\ref{lem: Lp for nabla u} and Proposition~\ref{propn: logY} reveals that powers of $\tilde{\rho}$ and $\hatrho$ of the form $\max\{0, \cdots\}$ arise from two inequalities, \eqref{NOV24, to use} and \eqref{NOV24, to use2} reproduced below:
\begin{eqnarray*}
&&\left\|\frac{P-\bar{P}}{\lambda+2\mu}\right\|_{L^q}  \leq C(q, E_0)\tilde{\rho}^{\max\left\{0,\frac{q-1}{q}\gamma-\beta\right\}}(t)\qquad\text{for any }q>2,\\
&&   \left\|\frac{P-\bar{P}}{(\lambda+2\mu)^{3/2}}\right\|_{L^{2+\frac{2\beta}{\gamma}}}\leq C(q, E_0)\tilde{\rho}^{\max\left\{0,\frac{1}{2}\gamma-\beta\right\}}(t).
\end{eqnarray*}
Nevertheless, in view of Lemma~\ref{lem:rhoLp} (Va\u{\i}gant--Kazhikhov, \cite[p.1119, Equation~(36)]{vaigant}) and that $P=\rho^\gamma$, we may instead bound $\left\|\frac{P-\bar{P}}{\lambda+2\mu}\right\|_{L^q} $ and $\left\|\frac{P-\bar{P}}{(\lambda+2\mu)^{{3}/{2}}}\right\|_{L^{2+\frac{2\beta}{\gamma}}}$ by a time-dependent constant $C(T,q,E_0,\gamma,\mu,\beta)$. This verifies Equations~\eqref{Lp2} and \eqref{logY2}.

\smallskip
\noindent
{\bf Step~2.} Next we give a time-dependent estimate for $\|G\|_\linf$, where $G$ is the commutator term as in Equation~\eqref{G, def}:
\begin{align*}
    G = \sum_{i,j \in \{1,2\}}    \big[u_i,(-\Delta)^{-1}\partial_i\partial_j\big](\rho u_j) \equiv \sum_{i,j \in \{1,2\}} \big[u_i, R_iR_j\big](\rho u_j).
\end{align*} 

To this end, fix any $q_1 \in]2,\infty[$ as in condition~($\clubsuit$) and $\alpha >\frac{4}{q_1-2}$. By using the Gagliardo--Nirenberg--Sobolev interpolation inequality, the commutator estimates in Lemma~\ref{commu} (\emph{cf}. Coifman--Rochberg--Weiss\cite{Coifman} and Coifman--Meyer \cite{Coifman1}), as well as the embedding $\dot{W}^{1,2}(\tor) \emb \mathcal{BMO}(\tor)$ (see \eqref{bmo2}), we deduce that 
 \begin{align}
 \label{upG2}
\|G\|_{L^{\infty}}&\leq C\|\nabla G\|_{L^{\frac{q_1(2+\alpha)}{q_1+\alpha+2}}}^{\frac{2(2+\alpha)}{q_1\alpha}}\|G\|_{L^{q_1}}^{1-\frac{2(2+\alpha)}{q_1\alpha}} \nonumber \\
&\leq C\|\nabla u\|_{L^{2+\alpha}}^{\frac{2(2+\alpha)}{q_1\alpha}}\|\rho u\|_{L^{q_1}}^{\frac{2(2+\alpha)}{q_1\alpha}}\|u\|_{{\mathcal{BMO}}}^{1-\frac{2(2+\alpha)}{q_1\alpha}}\|\rho u\|_{L^{q_1}}^{1-\frac{2(2+\alpha)}{q_1\alpha}}\nonumber \\
&\leq C\|\nabla u\|_{L^{2+\alpha}}^{\frac{2(2+\alpha)}{q_1\alpha}}\|\rho u\|_{L^{q_1}}^{\frac{2(2+\alpha)}{q_1\alpha}}\|\nabla u\|_{L^2}^{1-\frac{2(2+\alpha)}{q_1\alpha}}\|\rho u\|_{L^{q_1}}^{1-\frac{2(2+\alpha)}{q_1\alpha}}\nonumber \\
&= C\|\nabla u\|_{L^{2+\alpha}}^{\frac{2(2+\alpha)}{q_1\alpha}}\|\nabla u\|_{L^2}^{1-\frac{2(2+\alpha)}{q_1\alpha}}\|\rho u\|_{L^{q_1}}
\end{align}
with $C=C(q_1,\alpha)$.

On the right-most side of Equation~\eqref{upG2} we estimate $\|\nabla u\|_{L^{2+\alpha }}$ via \eqref{Lp2}, thus obtaining
\begin{equation}\label{Nov24, x}
\begin{aligned}
\|\nabla u\|_{L^{2+\alpha }}&\leq C \tilde{\rho} ^{\frac{\beta\varepsilon}{2}}\big[1+D\big] + C\tilde{\rho} ^{\frac{\beta\varepsilon}{2}+\frac{1}{2}-\frac{1}{2+\alpha}}\big[1+D\big]\left(\frac{X^2}{10+Y^2}\right)^{\frac{1}{2}-\frac{1}{2+\alpha}} + C,
\end{aligned}
\end{equation}
where $C=C(T, \alpha, \e , E_0, M, \gamma, \mu, \beta)$. By taking $q \in [1,\infty[$ such that $q_1=\frac{4q+1}{2q}$, we rewrite the condition~($\clubsuit$) in Theorem~\ref{thm: main} as
\begin{align}\label{Nov24, y}
        \|\rho u\|_{L^{\frac{4q+1}{2q}}}\leq C_\star \qquad\text{for a.e. } t \in [0,T].
     \end{align}

Substituting Equations~\eqref{Nov24, x} and \eqref{Nov24, y} into Equation~\eqref{upG2}, and making use of the div-curl estimate in Lemma~\ref{lemma1} and the definition of $D$ in \eqref{D, def}, we deduce that 
\begin{align}   \label{f1}
    \|G\|_{L^{\infty}}&\leq C\left\{\tilde{\rho} ^{\frac{\beta\varepsilon}{2}}\big[1+D\big]
    + \tilde{\rho} ^{\frac{\beta\varepsilon}{2}+\frac{1}{2}-\frac{1}{2+\alpha}}\big[1+D\big]\left(\frac{X^2}{10+Y^2}\right)^{\frac{1}{2}-\frac{1}{2+\alpha}} + 1\right\}^{\frac{2(2+\alpha)}{q_1\alpha}} D^{1-\frac{2(2+\alpha)}{q_1\alpha}}\nonumber\\
    &\leq C\left\{\tilde{\rho}^{\varepsilon}
    \left(D^{1-\frac{2(2+\alpha)}{q_1\alpha}}+D\right)+\tilde{\rho}^{\varepsilon+\frac{1}{q_1}}\left(D^{1-\frac{2(2+\alpha)}{q_1\alpha}}+D\right)\left( \frac{X^2}{10+Y^2}\right)^{\frac{1}{q_1}}\right\},
\end{align} 
with $C=C(T, \alpha, \e, q, E_0, M, \gamma, \mu, \beta, C_\star).$ 

\smallskip
\noindent
{\bf Step~3.} We now prove a time-dependent upper bound for $\hatrho^\beta$; compare with \S\ref{sec: rho bound}, Equation~\eqref{rho1}. 
\begin{align}\label{frho}\hatrho^{\beta}(t)
          \leq &\, C +\frac{\gamma}{2\mu\beta}\int_{0}^{t}e^{-\frac{\gamma}{2\mu}(t-\tau)}\hatrho^{\beta}(\tau)\,\mathrm{d}\tau+C\int_{0}^{t}e^{-\frac{\gamma}{2\mu}(t-\tau)}D(\tau)\,\mathrm{d}\tau\nonumber\\
          &+\int_{0}^{t} e^{-\frac{\gamma}{2\mu}(t-\tau)}\Bigg[\tilde{\rho}^{\varepsilon}
    \left(D^{1-\frac{2(2+\alpha)}{q_1\alpha}}(\tau)+D(\tau)\right)\nonumber\\
          &\qquad+\tilde{\rho}^{\varepsilon+\frac{1}{q_1}}\left(D^{1-\frac{2(2+\alpha)}{q_1\alpha}}(\tau)+D(\tau)\right)\left( \frac{X^2(\tau)}{10+Y^2(\tau)}\right)^{\frac{1}{q_1}}\Bigg]\,\mathrm{d}\tau\qquad\text{ for a.e. } t \in [0,T],
\end{align}
with some constant $C=C(T, \alpha, \e, q, E_0, M, \gamma, \mu, \beta, C_\star)$.

This shall be established via an analogous approach to the derivation of Equation~\eqref{rho1}, by utilising the estimates~ \eqref{f1}--\eqref{f3}. For the sake of brevity, we only point out the differences.

\begin{itemize}
    \item 
In place of Equation~\eqref{xxx, Nov24}, we bound for \emph{a.e.} $t \in [0,T]$ that
\begin{align}\label{f2}
\left\|(-\Delta)^{-1}\mathrm{div}(\rho u)\right\|_{L^{\infty}}&\leq C\|(-\Delta)^{-1}\nabla\mathrm{div}(\rho u)\|_{L^{\frac{4q+1}{2q}}} \nonumber\\
& \leq C\|\rho u\|_{L^{\frac{4q+1}{2q}}}\nonumber\\
&\leq C\Big(T, \alpha, \e, q, E_0, M, \gamma, \mu, \beta, C_\star\Big). 
\end{align}
Here we use the fact that  the average of $\rho u$ is zero, the Sobolev inequality, continuity of the Riesz operator, and the estimate~\eqref{Nov24, y}.

\item 
In place of Equation~\eqref{pb}, we invoke Lemma~\ref{lem:rhoLp} to deduce for \emph{a.e.} $t \in [0,T]$ that
\begin{equation}
   \label{f3}
\left|\bar{P}-\bar{B}\right|\leq  C\left[1+D(t)\right],
\end{equation}
with $C$ having the same dependence as the constant in \eqref{f2} above.
\end{itemize}

\smallskip
\noindent
{\bf Step~4.} We are now ready to conclude the essential upper bound for $\rho$ up to any finite time $T$. Indeed, we may infer from Equations~\eqref{frho} and \eqref{logY2}, H\"{o}lder inequality, and the energy inequality in Lemma~\ref{lemma: energy ineq} that 
\begin{align*}
\hatrho^{\beta}(t) &\leq C+\frac{\gamma}{2\mu\beta}\int_{0}^{t}e^{-\frac{\gamma}{2\mu}(t-\tau)}\hatrho^{\beta}(\tau)\,\mathrm{d}\tau+C\left\{\int_{0}^{t}e^{-\frac{\gamma}{2\mu}(t-\tau)}D^2(\tau)\,\mathrm{d}\tau\right\}^{\frac{1}{2}}\\
&\quad+C\hatrho^{\varepsilon}(t)\left\{\int_{0}^{t} e^{-\frac{\gamma}{2\mu}(t-\tau)}D^{2}(\tau)\,\mathrm{d}\tau\right\}^{\frac{q_1\alpha-4-2\alpha}{2q_1\alpha}}+C\hatrho^{\varepsilon}(t)\left\{\int_{0}^{t} e^{-\frac{\gamma}{2\mu}(t-\tau)}D^{2}(\tau)\,\mathrm{d}\tau\right\}^{\frac{1}{2}}\\
&\quad+C\hatrho^{\varepsilon+\frac{1}{q_1}}(t)\left\{\int_{0}^{t} e^{-\frac{\gamma}{2\mu}(t-\tau)}D^{2}(\tau)\,\mathrm{d}\tau\right\}^{\frac{q_1\alpha-4-2\alpha}{2q_1\alpha}}\left\{\int_{0}^{t} e^{-\frac{\gamma}{2\mu}(t-\tau)}\frac{X^2(\tau)}{10+Y^2(\tau)}\,\mathrm{d}\tau\right\}^{\frac{1}{q_1}}\\
&\quad+C\hatrho^{\varepsilon+\frac{1}{q_1}}(t)\left\{\int_{0}^{t} e^{-\frac{\gamma}{2\mu}(t-\tau)}D^{2}(\tau)\,\mathrm{d}\tau\right\}^{\frac{1}{2}} \left\{\int_{0}^{t} e^{-\frac{\gamma}{2\mu}(t-\tau)}\frac{X^2(\tau)}{10+Y^2(\tau)}\,\mathrm{d}\tau\right\}^{\frac{1}{q_1}}\\
&\leq C+\frac{\gamma}{2\mu\beta}\int_{0}^{t}e^{-\frac{\gamma}{2\mu}(t-\tau)}\hatrho^{\,\beta}(\tau)\,\mathrm{d}\tau +C\hatrho^{\,\,\varepsilon+\frac{2}{q_1}}(t)\qquad \text{for a.e. } t \in [0,T].
\end{align*}
Here, as before, the constant $C=C(T, \alpha, \e, q, E_0, M, \gamma, \mu, \beta, C_\star)$. 

Recall $q_1 = \frac{4q+1}{2q}$; hence, $\varepsilon +\frac{2}{q_1} =\varepsilon +\frac{4q}{4q+1}$. If this parameter can be chosen to be less than $\beta$, then by Gr\"{o}nwall's lemma $\hatrho^{\,\beta}(t)\leq C$. Indeed, choose $\e=\e(q)$ such that $\e + \frac{2}{q_1} \leq 1$ where $\beta>1$ by assumption, and choose $\alpha>\frac{4}{q_1-2}$. By fixing $q$, $\alpha$, and $\e$ once and for all, we obtain a constant $\bar{M}(T)<\infty$ depending only on $T$, $E_0$, $M$, $\gamma$, $\mu$, $\beta$, $q_1$, and $C_\star$, such that 
\begin{equation}
    \label{uprho2}
    \rho(t,x)\leq \bar{M}(T),\qquad \text{ for a.e. } (t,x)\in [0,T]\times\mathbb{T}^2.
\end{equation}
This concludes the proof of the upper bound for density.

\smallskip
\noindent
{\bf Step~5.} Finally, let us prove the essential lower bound for $\rho$ up to any finite time $T$ away from zero, via an adaptation of the arguments in \S\ref{lower bound}.

Recall the definition of $\theta$ in \eqref{theta, def} and the bound~\eqref{integrated diff ineq for theta} it satisfies:
\begin{equation*}
\begin{cases}
\theta(\rho):= 2\mu \log\rho +\beta^{-1}\rho^{\beta},\\
\theta\big(\rho(t,x)\big)\geq\theta\big(\rho_0(x)\big) -\left\|(-\Delta)^{-1}\mathrm{div}(\rho u)\right\|_{L^{\infty}}\nonumber +\int_{0}^{t} \left\{\bar{P}-\big|G+P+\bar{B}\big|\right\}(\tau, x)\,\mathrm{d}\tau.
\end{cases}
\end{equation*}
The second term on the right-hand side of the inequality below is bounded via \eqref{f2}.

For the integral term $\int_{0}^{t} \left\{\bar{P}-\big|G+P+\bar{B}\big|\right\}(\tau, x)\,\mathrm{d}\tau$, in view of Equations~\eqref{f1} and \eqref{uprho2} we have that
\begin{align*}
\sup_{x\in \mathbb{T}^2}\left|G(t,x)\right| \leq C\left[
    \left(D^{1-\frac{2(2+\alpha)}{q_1\alpha}}+D\right)+\left(D^{1-\frac{2(2+\alpha)}{q_1\alpha}}+D\right)\left( \frac{X^2}{10+Y^2}\right)^{\frac{1}{q_1}}\right].
\end{align*}
Again $q_1 = \frac{4q+1}{2q}$. On the other hand, by Equation~\eqref{uprho2} and Cauchy--Schwarz,
      $$\begin{aligned}
        \sup_{x\in \mathbb{T}^2}|P(t,x)|+\left|\bar{B}\right|&\leq C+\left\{\int_{\mathbb{T}^2} (\lambda+2\mu )\,\mathrm{d}x\right\}^{\frac{1}{2}}\left\{\int_{\mathbb{T}^2} (\lambda+2\mu)|\mathrm{div}u|^2 \,\mathrm{d}x\right\}^{\frac{1}{2}}\\
        &\leq C\big(1+D(t)\big).
      \end{aligned}$$
We may thus use Young's inequality to obtain  that
    \begin{equation*}
\sup_{x\in\mathbb{T}^2}\left|G(t,x)+P(t,x)+\bar{B}\right|\leq C\left\{1+ D^2(t)+\left(\frac{X^2}{10+Y^2}\right)(t)\right\}\qquad\text{for any $\delta>0$}.
    \end{equation*}
Then, from the energy inequality in Lemma~\ref{lemma: energy ineq} and Equation~\eqref{log Y bound}, one infers that
\begin{align}\label{Nov24, lower bd}
 \int_{0}^{t} \left\{\bar{P}-\left|G+P+\bar{B}\right|\right\}\,\mathrm{d}\tau 
&\geq -\int_{0}^{t} \left\{\sup_{x\in\mathbb{T}^2}\left|G+P+\bar{B}\right|\right\}\,\mathrm{d}\tau\nonumber\\     
&\geq -C\int_{0}^{t}\left\{D^2(\tau)+\left(\frac{X^2(\tau)}{10+Y^2(\tau)}\right)\right\} \,\mathrm{d}\tau\nonumber\\
&\geq -C\qquad\text{for a.e. } t \in [0,T].
\end{align}
Here, in view of the dependence on parameters of $\bar{M}(T)$ in Equation~\eqref{uprho2}, we note that the constant $C$ here depends only on $T$, $E_0$, $M$, $\gamma$, $\mu$, $\beta$, $q_1$, and $C_\star$.

We may now conclude the proof of the lower bound for $\rho$ as in \S\ref{lower bound}. Indeed, by substituting Equations~\eqref{f2} and \eqref{Nov24, lower bd} into \eqref{integrated diff ineq for theta}, one obtains that
    \begin{equation*}
\theta\big(\rho(t,x)\big)\geq \theta \big(\rho_0(x)\big)-C\Big(T, \nu, \alpha, E_0, M, \gamma, \mu, \beta\Big)\qquad\text{for \emph{a.e.}} (t,x) \in [0,T]\times\tor.
    \end{equation*}
As argued in Perepelitsa \cite{perep}, when $\rho$ becomes close to zero (depending additionally on the positive constant $\bar{m}$, the essential lower bound for initial density $\rho_0$) at time $t$, $\theta(\rho) \approx \log \rho$  above inequality can relax to $\log \rho$, thus prohibiting the formation of vacuum. Therefore, for some constant $\bar{m}(T)>0$ it holds that
	\begin{equation}
        \label{lowrho2}
		\rho(t,x)\geq \bar{m}(T)\qquad\text{ for a.e. } (t,x)\in [0,T]\times \mathbb{T}^2,
\end{equation}
where $\bar{m}(T)=\bar{m}(T, E_0, M, m, \gamma, \mu, \beta, q_1, C_\star)$.

\smallskip
The proof of Proposition~\ref{propn: density bound, beta>1} is now complete in view of Steps~1--5.    \end{proof}

\begin{remark}
The additional condition~($\clubsuit$) has only been invoked in Step~2 of the proof above. In an earlier attempt, we tried to bypass the condition~($\clubsuit$) by invoking the time-dependent integral bound for $\rho|u|^{2+\nu}$ in Huang--Li \cite[Equation~(3.50) in the proof of Lemma~3.7]{H1}:
 \begin{lemma*} 
Assume $\beta>1$. There is a positive constant $\nu_0\leq \frac{1}{2}$ depending only on $\mu$ such that
     $$\sup_{t\in [0,T]}\int_{\mathbb{T}^2}\rho|u|^{2+\nu}\,\mathrm{d}x \leq C,$$
     where $\nu=\hatrho(T)^{-\frac{\beta}{2}}\nu_0$. Here  $C$ depends only on $T$, $\mu$, $\beta$, $\gamma$, $\|\rho_0\|_{L^{\infty}}$, and $\|u_0\|_{L^{\infty}}$ as in Lemma~\ref{lem:rhoLp}. 
 \end{lemma*}
 The issue, however, is that the index $\nu$ in turn depends on the density upper bound. Indeed, if we apply Lemma~\ref{lem:rhoLp} to $q_1=\frac{4q+1}{2q}$ where $q \in \left](2\nu)^{-1},\infty\right[$, we obtain by H\"{o}lder's inequality that
\begin{align*}
       &\|\rho u\|_{L^{\frac{4q+1}{2q}}}=\left\{ \int_{\tor} \rho^{\frac{1+\nu}{2+\nu}\cdot \frac{4q+1}{2q}} \left(\rho^{\frac{1}{2+\nu}}|u| \right)^{\frac{4q+1}{2q}}\,\dd x\right\}^{\frac{2q}{4q+1}}\nonumber\\
        &\leq \|\rho\|_{L^{\frac{(1+\nu)(4q+1)}{2q\nu -1}}}^{\frac{1+\nu}{2+\nu}}\left(\int_{\tor}\rho |u|^{2+\nu}\,\mathrm{d}x\right)^{\frac{1}{2+\nu}}.
     \end{align*}
But $\nu \sim \widehat{\rho}(T)^{-\beta/2}$ and hence $\frac{(1+\nu)(4q+1)}{2q\nu -1}\lesssim\frac{4q}{2q\nu-1}\lesssim \nu^{-1}$, so we have
\begin{align*}
\|\rho\|_{L^{\frac{(1+\nu)(4q+1)}{2q\nu -1}}}^{\frac{1+\nu}{2+\nu}} \lesssim \left[\left(\frac{1}{\nu}\right)^{\frac{2}{\beta-1}}\right]^{\frac{1+\nu}{2+\nu}} \lesssim \left(\frac{1}{\nu}\right)^{\frac{1}{\beta-1}} \lesssim \widehat{\rho}(T)^{\frac{2}{\beta-1}},
\end{align*} 
which does not yield a bound for $\|\rho u\|_{L^{q_1}}$ independent of the upper bound for $\rho$. 
 \end{remark}

\subsection{Global existence of weak solutions when $\beta>1$}\label{sec: final}

Finally, with Proposition~\ref{propn: density bound, beta>1} at hand, we are at the stage of deducing the global existence of weak solutions under the mere assumption that $\beta>1$ and $\gamma>1$. Our argument is motivated by Va\u{\i}gant--Kazhikhov  \cite[pp.1136--1139]{vaigant}, which in turn relies on the classical work of Yudovich \cite{yudovich}.

To begin with, we mollify the initial data $(\rho_0, u_0)$ to obtain a sequence $\left\{\left(\rho_0^n, u_0^n\right)\right\}$ such that 
\begin{equation*}
    \begin{cases}
0<m<\rho_0^n(x)<M<\infty\qquad\text{ for all } x\in\mathbb{T}^2,\\
(\rho_0^n,u_0^n)\in C^{1+\omega}\left(\mathbb{T}^2\right)\times C^{2+\omega}\left(\mathbb{T}^2;\R^2\right)\qquad\text{for some } 0<\omega<1,\\
(\rho_0^n,u_0^n) \longrightarrow (\rho_0, u_0)\qquad\text{ in $L^\infty \times H^1$ as $n \to \infty$.}
    \end{cases}
\end{equation*}
We shall fix $\omega$ once and for all in the sequel. By Theorem \ref{thm: VK} (Va\u{\i}gant--Kazhikhov  \cite{vaigant}), there exists for each $n \in \mathbb{N}$ a unique global classical solution $(\rho^n, u^n)$ for Equation~\eqref{equ}.

The arguments in the earlier parts of the paper give us the following:
 \begin{itemize}
    \item The energy inequality (see Lemma~\ref{lemma: energy ineq}).
\begin{equation*}
\begin{aligned}
&\sup\limits_{t>0} \int_{\mathbb{T}^2} \left\{\rho^n(t,x)\frac{|u^n(t,x)|^2}{2} +\frac{P(\rho^n(t,x))}{\gamma-1}\right\}\,\mathrm{d}x \\ &\qquad+\iint\limits_{]0,+\infty[\times \mathbb{T}^2} \bigg\{\big(\lambda(\rho^n)+2\mu\big)|\mathrm{div} \,u^n(t,x)|^2+\mu|\mathrm{rot}\, u^n(t,x)|^2\bigg\}\,\mathrm{d}x\,\mathrm{d}t\leq C.
\end{aligned}
\end{equation*}
\item The bounds for density (see Equations~\eqref{uprho2} and \eqref{lowrho2}): there exist $0<\bar{m}(T) <\bar
{M}(T)<\infty$ such that for any $T \in ]0,\infty[$, it holds that
\begin{equation*}
    \bar{m}(T)<\rho^n <\bar{M}(T)\qquad \text{for all } (t,x)\in [0,T]\times \mathbb{T}^2.
\end{equation*}
\item  In view of Equation~\eqref{uprho2}, the integral bound for $D^2(t) + Y^2(t)$, and the previous bound for $\rho^n$, we have that
\begin{equation*}
\int_{0}^{T}\left\{\left\|B^n\right\|_{H^1}^2+\left\|\mathrm{rot} \,u^n\right\|_{H^1}^2\right\} \,\mathrm{d}t\leq C.
 \end{equation*}
 \end{itemize} 
In the above, $\left(\rho^n, u^n\right)$ are the classical solution corresponding to the initial data $\left(\rho_0^n,u_0^n\right)$. The quantities $B^n$, $D^n$, and $Y^n$ are defined as in Equations~\eqref{defB}, \eqref{D, def}, and \eqref{Y, def}, respectively, with  $\left(\rho^n, u^n\right)$  in lieu of $(\rho, u)$. Moreover, the constants $C$, $\bar{m}(T)$, and $\bar{M}(T)$ depend only on $T$, $E_0$, $M$, $m$, $\gamma$, $\mu$, and $\beta$.

Thus, one may extract a subsequence from $\left\{\left(\rho^n,u^n\right)\right\}$ (without relabelling) such that
\begin{equation}
   \label{lim}
    \begin{aligned}
        \rho^n \weak \rho \qquad &\text{weakly-$\star$ in } \quad L^{\infty}\big([0,T]\times \mathbb{T}^2\big);\\
  u^n \weak u \qquad &\text{weakly in} \quad L^{2}\left(0,T; H^1(\mathbb{T}^2; \R^2) \right);\\
  B^n \weak B \qquad &\text{weakly in} \quad L^{2}\left(0,T; H^1(\mathbb{T}^2)\right);\\
 \mathrm{rot}\, u^n \weak \mathrm{rot}\, u\qquad &\text{weakly in} \quad L^{2}\left(0,T; H^1(\mathbb{T}^2)\right).
    \end{aligned}
\end{equation}
Let $\Phi: \left[0, \bar{M}(T)+1\right] \to \R$ be an arbitrary smooth function. Thanks to the first line in \eqref{lim} above, there exists $\limphi \in \linf\big([0,T]\times\tor\big)$ such that modulo subsequences (unrelabelled),
\begin{equation*}
    \Phi\left(\rho^n\right) \weak \limphi \qquad\text{weakly-$\star$ in } \quad L^{\infty}\big([0,T]\times \mathbb{T}^2\big).
\end{equation*}

Denote 
\begin{equation*}
\begin{cases}
\xi(s) := \frac{1}{\lambda(s)+2\mu} \qquad\text{ for } s \in \left[0, \bar{M}(T)+1\right],\\
P_1(\rho) := P(\rho) - \bar{P}(\rho).
\end{cases}
\end{equation*}
It then holds that 
\begin{align*}
    {\rm div}\, u^n = \xi\left(\rho^n\right) B^n + \xi\left(\rho^n\right) P_1\left(\rho^n\right).
\end{align*}
In view of \eqref{lim} and the usual Sobolev embedding, we see that $B^n \to B$ strongly in $L^2\big(0,T; L^q(\tor)\big)$ for any $q \in [1,\infty[$. Hence, 
\begin{align*}
 \mathrm{div} u=\bra\xi\ket B+\bra{\xi P_1}\ket.
\end{align*}

To show that $(\rho, u)$ is indeed a weak solution in $[0,T]\times\tor$, as argued at the beginning of Va\u{\i}gant--Kazhikhov  \cite[p.1137]{vaigant}, it suffices to establish that 
\begin{align*}
    \bra\xi\ket =\xi(\rho)\qquad\text{ and }\qquad \bra\xi P_1\ket = \xi(\rho)P_1(\rho).
\end{align*}
As in Lions \cite{lions'} and Hoff \cite{hoff2} (see also \cite[p.1137]{vaigant}), this is warranted by showing that $\limphi = \Phi(\rho)$ for any fixed strictly convex function $\Phi$. We shall take $\Phi(s)=s^2$ in the sequel. That is, it remains to show that
\begin{align}\label{goal, rhos}
    \Psi := \rhos - \rho^2 = 0.
\end{align}
Thanks to the convexity of $\Phi(s)=s^2$, we already have that
\begin{equation}\label{Psi geq 0}
    \Psi \geq 0.
\end{equation}

As per Va\u{\i}gant--Kazhikhov  \cite[Equations~(99), (100)]{vaigant} and Huang--Li \cite[Equation~(5.14) and the preceding paragraph]{H1}, by both viewing $\rho$ as a renormalised solution and by directly passing to the limits for the classical solution $(\rho^n, u^n)$, respectively, one obtains
\begin{equation*}
    \begin{cases}
        \frac{\partial \rho^2}{\partial t}+\mathrm{div}\left(\rho^2 u\right)+ B\rho^2 
        \bra\xi\ket+\rho^2\bra{\xi P_1}\ket=0,\\
        \\
        \frac{\partial \bra{\rho^2}\ket}{\partial t}+\mathrm{div}\left(\bra{\rho^2}\ket u\right)+ B\bra{\rho^2\xi}\ket+\bra{\rho^2\xi P_1}\ket=0,
    \end{cases}
\end{equation*}
both in $\mathcal{D}^{\prime}\big(]0,T[\times \mathbb{T}^2\big)$. 
It follows that
\begin{equation}
  \label{Psi}
    \frac{\partial \Psi}{\partial t}+\mathrm{div}(\Psi u)+ B\left(\bra\rho^2\xi\ket - \rho^2 \bra\xi\ket\right)+ \bra \rho^2\xi P_1\ket - \rho^2 \bra\xi P_1\ket=0\quad \text{ in }\mathcal{D}^{\prime}\big(]0,T[\times \mathbb{T}^2\big),
\end{equation}
equipped with the initial condition $\Psi\big|_{t=0}=0$ \emph{a.e.} in $\tor$.

To proceed, we \emph{claim} that for some constant $C=C(T, E_0, M, m, \gamma, \mu, \beta)$, it holds that
\begin{equation}\label{claim, Nov24}
\int_{\mathbb{T}^2}\Psi \,\mathrm{d}x \leq C\int_0^{t}\int_{\mathbb{T}^2}\big(|B|+1\big)\Psi\,\mathrm{d}x\mathrm{d}\tau\qquad \text{for a.e. } t \in [0,T].
\end{equation}
Indeed, noticing that 
\begin{equation*}
 0< \frac{1}{\left(\bar{M}+1\right)^\beta + 2\mu} \leq \xi(\rho^n) \leq \frac{1}{2\mu} <\infty
\end{equation*}
and using the definition of $\bra\bullet\ket$ and $\Psi$, we obtain for some $C=C(\bar{M}, T, \mu, \beta)$ that
\begin{align*}
    0 &\leq \left|\int_\tor \varphi \left\{\bra\rho^2\xi\ket - \rho^2 \bra\xi\ket\right\}\,\dd x\right| \nonumber\\
    &\qquad\qquad + \left|\int_\tor \varphi \left\{\bra \rho^2\xi P_1\ket - \rho^2 \bra\xi P_1\ket\right\}\,\dd x \right|\leq C \int_\tor \left(|\varphi|+1\right)\Psi\,\dd x\qquad\text{for each $\varphi \in L^1\left(\tor\right)$}.
\end{align*}
The \emph{claim}~\eqref{claim, Nov24} then follows from integrating both sides of Equation~\eqref{Psi} over $[0,t]\times \tor$.

We are now ready to conclude via a Gr\"{o}nwall argument similar in spirit to  Yudovich  \cite{yudovich}. Utilising the above \emph{claim}~\eqref{claim, Nov24} and H\"{o}lder's inequality, one deduces that
\begin{align}\label{Z1, Nov24}
\int_{\mathbb{T}^2}\Psi (t,x)\,\mathrm{d}x &\leq \int_0^{t}\big\| |B|+1 \big\|_{L^{\frac{2}{\varepsilon}}}\left\|\Psi^{1-\e}\right\|_{L^{\frac{\e}{1-\e}}}\|\Psi^\e\|_{L^{\frac{2}{\e}}}\,\mathrm{d}\tau\nonumber \\
&\leq \int_0^{t}\big\| |B|+1 \big\|_{L^{\frac{2}{\varepsilon}}}\|\Psi\|_{L^1}^{1-\varepsilon}\|\Psi\|_{L^{2}}^{\varepsilon}\,\mathrm{d}\tau \nonumber\\
&=: \mathcal{Z}(t)
\end{align}
for any $\e \in ]0,1[$. One thus obtains the differential inequality
\begin{equation*}
    \frac{\mathrm{d} \mathcal{Z}(t)}{\mathrm{d}t}\leq \big\||B|+1\big\|_{L^{\frac{2}{\varepsilon}}}\|\Psi\|_{L^{2}}^{\varepsilon}  \mathcal{Z}^{1-\varepsilon}(t),
\end{equation*}
which implies that
\begin{align}\label{Z, Nov24}
\mathcal{Z}^{\varepsilon}(t)&\leq \varepsilon \int_0^{T}\big\||B|+1\big\|_{L^{\frac{2}{\varepsilon}}}\|\Psi\|_{L^{2}}^{\varepsilon}\,\mathrm{d}t\nonumber\\
        &\leq \varepsilon\left(\int_0^{T}\big\| |B|+1\big\|_{L^{\frac{2}{\varepsilon}}}^2\,\mathrm{d}t\right)^{\frac{1}{2}}\left(\int_0^{T}\|\Psi\|_{L^2}^{2\varepsilon}\,\mathrm{d}t\right)^{\frac{1}{2}}
\end{align}
by integrating $\dd \mathcal{Z} \slash \mathcal{Z}^{1-\e} = \e^{-1}\dd\left(\mathcal{Z}^\e\right)$ over time and applying the Cauchy--Schwarz inequality.

To control the right-hand side of \eqref{Z, Nov24}, note that by the time-dependent upper bound~\eqref{uprho2} for $\rho$, we have $\|\Psi\|_{L^2} \leq C(T, E_0, M,m, \gamma, \mu, \beta, q_1, C_\star)$. On the other hand, by virtue of Lemma~\ref{1} and the ensuing Remark~\ref{remark: const} (observe the dependence of the constant on $\e$), we have that 
\begin{align*}
    \big\| |B|+1\big\|_{L^{\frac{2}{\e}}} \leq C\sqrt{{2}\slash{\e}} \, \big\| |B|+1\big\|_{L^2}^\e \|B\|^{1-\e}_{H^1}. 
\end{align*}
Furthermore, 
\begin{align*}
    \big\| |B|+1\big\|_{L^2} \quad\text{and}\quad  \| B\|_{H^1} \leq C\Big(T, E_0, M,m, \gamma, \mu, \beta, q_1, C_\star\Big).
\end{align*}
To see this, we first invoke the pointwise-in-time estimate~\eqref{BH1} for $\|B\|_{H^1}$ in terms of the upper bound for ${\rho}$, as well as $D(t)$ and $Y(t)$. Here we have an upper bound for $\rho$ (depending on $T$; see Equation~\eqref{uprho2}), and hence $\int_0^T \left\{D^2(t)+Y^2(t)\right\}\,\dd t \leq C <\infty$ for some constant $C$ depending on $T$ as well, via an adaptation of Equations~\eqref{D2, Nov24} and \eqref{log Y bound}.

The preceding argument allows us to infer from Equation~\eqref{Z, Nov24} the following bound:
\begin{equation}\label{final, Nov24}
\mathcal{Z}^{\varepsilon}(t) \leq C'\sqrt{\e} \cdot C^\e\quad\text{for both $C$ and $C'$ depending only on $T, E_0, M,m, \gamma, \mu, \beta, q_1, C_\star$.}
\end{equation}
As $\lim_{\e \searrow 0} \sqrt{\e}^{1/\e} = 0$, we conclude that $\mathcal{Z}(t)=0$. But by Equation~\eqref{Z1, Nov24} we have $\int_\tor \Psi(t,x)\,\dd x \leq \mathcal{Z}(t)$ and by Equation~\eqref{Psi geq 0} we have $\Psi \geq 0$. This proves the desired identity~\eqref{goal, rhos}.

Moreover, with the uniform bounds for density $\rho$ in Proposition~\ref{propn: density bound, beta>1}, we may adapt the higher-order estimates for $u$ in Huang--Li~\cite{H1} to deduce that $\|\na u\|_{L^\infty_t L^r_x([1,T] \times \tor)} \leq C(r, C_\star)$  for any $r \geq 2$. Here $C_\star$ is as in the condition~($\clubsuit$). Since $\|\na u(t,\bullet)\|_{L^2(\tor)}\to 0$, and interpolation shows that $\|\na u(t,\bullet)\|_{L^p(\tor)}\to 0$ for any $p \in [2,r]$.  
Since $r \geq 2$ is arbitrary, this completes the proof of the Main Theorem~\ref{thm: main}.

\medskip
\noindent
{\bf Acknowledgement}. 
SL is indebted to Professors Tao Wang and Zhuan Ye for their kind communications and for pointing out some mistakes in earlier versions of the manuscript.

The research of SL is supported by NSFC Projects 12201399, 12331008, and 12411530065, Young Elite Scientists Sponsorship Program by CAST 2023QNRC001, National Key Research $\&$ Development Programs 2023YFA1010900 and 2024YFA1014900, Shanghai Rising-Star Program 24QA2703600, Shanghai Qi-Guang Scholarship, and the Shanghai Frontier Research Institute for Modern Analysis. The research of JY is partially supported by National Key Research $\&$ Development Programs 2023YFA1010900 and 2024YFA1014900.

\smallskip
\noindent
{\bf Conflict of Interest Statement}.  On behalf of all authors, the corresponding author states that there is no conflict of interest.

\smallskip
\noindent
{\bf Data Availability Statement}. No data, models or code were generated or used during the study.

\end{document}